\documentclass[11pt]{article}

\usepackage{amsmath, amsfonts,amsthm,amssymb,bm}
\usepackage{dsfont}
\usepackage{pifont}
\usepackage[english]{babel}

\usepackage[normalem]{ulem}

\DeclareRobustCommand*{\ora}{\overrightarrow}

\usepackage{latexsym,amssymb,amsfonts,graphicx}
\usepackage{amsmath,dsfont}
\usepackage{verbatim}
\usepackage{mathrsfs}
\usepackage{bm}
\usepackage{color}
\usepackage{epsfig}
\usepackage{float}

\usepackage{url}
\usepackage{bm}
\usepackage{natbib}

\usepackage[colorlinks,linkcolor=blue,citecolor=blue,anchorcolor=blue]{hyperref}

\newcommand{\Px}{ \mathbb{P} }

\newcommand{\Ex}{ \mathbb{E} }

\def\esssup_#1{\underset{#1}{\mathrm{ess\,sup\, }}}
\def\essinf_#1{\underset{#1}{\mathrm{ess\,inf\, }}}
\def\argmax_#1{\underset{#1}{\mathrm{arg\,max\, }}}
\def\argmin_#1{\underset{#1}{\mathrm{arg\,min\, }}}

\newcommand{\Fx}{\mathbb{F} }

\newcommand{\F}{\mathcal{F}}

\newcommand{\R}{\mathbb{R}}

\newtheorem{theorem}{Theorem}[section]

\numberwithin{equation}{section}

\newtheorem{proposition}[theorem]{Proposition}

\newtheorem{lemma}[theorem]{Lemma}

\allowdisplaybreaks[4]

\definecolor{Red}{rgb}{1.00, 0.00, 0.00}

\definecolor{DRed}{rgb}{0.5, 0.00, 0.00}

\definecolor{Blue}{rgb}{0.00, 0.00, 1.00}

\definecolor{Green}{rgb}{0.0, 0.4, 0.0}

\title{De Finetti's Poissonian Dividend Control Problem under Spectrally Positive Markov Additive Process}
\author{Lijun Bo\thanks{School of Mathematics and Statistics, Xidian University, Xi'an, 710126, China, and School of Mathematical Sciences, University of Science and Technology of China, Hefei, 230026, China. Email: lijunbo@xidian.edu.cn}\,\,\,\,\,\,\,\,Wenyuan Wang\thanks{Corresponding author. School of Mathematics and Statistics, Fujian Normal University, Fuzhou, 350007, P.R. China; School of Mathematical Sciences, Xiamen University, Xiamen, 361005, P.R. China. Email: wwywang@xmu.edu.cn}\,\,\,\,\,\,\,\,
Kaixin Yan\thanks{School of Mathematical Sciences, Xiamen University, Xiamen, Fujian, China. Email: kaixinyan@stu.xmu.edu.cn}
}
\date{}

\begin{document}
\maketitle

\begin{abstract}
We study a De Finetti's optimal dividend and capital injection problem under a Markov additive model. The surplus process without dividend and capital injection is assumed to follow a spectrally positive Markov additive process (MAP). Dividend payments are made at the jump times of an independent Poisson process and capitals are injected to avoid bankruptcy. The aim of the paper is to characterize an optimal dividend and capital injection strategy that maximizes the expected total discounted dividends subtracted by the total discounted costs of capital injection. Applying the fluctuation and excursion theory for L\'evy processes and the stochastic control theory, we first address an auxiliary dividend and capital injection control problem with a terminal payoff under the spectrally positive L\'evy model. Using results obtained for this auxiliary problem and a fixed point argument for iterations induced by dynamic program, we characterize the optimal strategy of our prime control problem as a regime-modulated double-barrier Poissonian-continuous-reflection dividend and capital injection strategy. Besides, a numerical example is provided to illustrate the features of the optimal strategies. The impacts of model parameters are also studied.
\ \\

\end{abstract}

\section{Introduction}\label{sec:intro}

Paying dividends to investors is a common policy in the economics of corporate finance (\cite{Feldstein83}). The De Finetti's dividend problem amounts to a kind of stochastic optimal control problem that aims to identify an optimal dividend strategy that maximizes the expected total discounted net dividends of an insurance firm. It was first studied in \cite{De Finetti57} with the surplus process of an insurance firm described as a symmetric random walk. The author proves that the optimal dividend strategy is in fact a single barrier one.

Following the pioneer work by \cite{De Finetti57}, research works along this line have been growing fast in the context of insurance and corporate finance. To name a few, the literature has been witnessing many progresses when the underlying surplus process is the Cram\'er-Lundberg process (see e.g. \cite{Kulenko08}, \cite{Albrecher11}, \cite{Avanzi11}, \cite{Azcue15}, \cite{Vierkotter17}), Brownian diffusion processes (see e.g. \cite{Jiang12}, \cite{Yang16}
) and spectrally one-sided or two-sided L\'evy processes (see e.g. \cite{Avram07}, \cite{Bayraktar13}, \cite{Zhao17}, \cite{Noba18}, \cite{Perez18}, \cite{Noba21}, \cite{Wang22.b}). On the other hand, substantial dividend payments will inevitably increase the risk exposure of bankruptcy for the company. Hence, to protect the company, it is a common business to raise new capital, which is termed as capital injection in the literature, by issuing new debts that usually appears in the form of loan from the bank (see e.g. \cite{Easterbrook84}).

It is known that the Markov additive process (MAP) is of regime-switching models, which can be viewed as a family of L\'evy processes switching according to an independent continuous-time Markov chain with finite states. The regime-switching model has been widely studied due to its capability to capture the transitions of market behaviors and its mathematical tractability and explicit structures. A list of empirical studies concerning regime-switching can be found in \cite{Ang12}, \cite{Hamilton89} and \cite{So98}. However, as far as we know, there are very few works that investigate the De Finetti's optimal dividend problem with capital injection when the surplus process is described as a spectrally one-sided MAP. In fact, the only three ones among the literature can be found in \cite{Noba20}, \cite{Mata22} and \cite{Wang22.c}. More specifically, \cite{Noba20}, under a spectrally negative MAP, study a class of De Finetti's dividend and capital injection problems subject to the constraint that the accumulated dividend process is absolutely continuous with bounded density, and the optimality of a regime-modulated refraction-reflection strategy is verified. As an extension, \cite{Wang22.c}, under a spectrally positive MAP, handle the case where the accumulated dividend process can be any non-decreasing, right-continuous and adapted process. A pair of strategy of regime-modulated double barrier dividend and capital injection is proved to dominate any other admissible strategy. The works reviewed above is based on a continuous dividend decisions made by the firm. For the case where the decision maker makes dividend decisions at independent Poisson arrival times, \cite{Mata22} consider Poisson observation version of De Finetti's dividend and capital injection problem under a spectrally negative MAP. By introducing an auxiliary problem, they provide the optimality of a regime-modulated double barrier dividend and capital injection strategy.


On the other hand, the spectrally
negative L\'evy process (SNLP) has now been widely accepted as a perfect replacement of the classical Cram\'er-Lundberg process to model the surplus process of an insurance company. The spectrally
positive L\'evy process (SPLP), as a skip-free downwards model, is better suited for modelling the surplus processes that pay expenses continuously and have gains occasionally, e.g., the capital of a business engaging
in research that pays daily expenditure for research purpose and gain occasional profits with random amounts due to patent awards or sudden sales increase (see e.g. \cite{Avanzi11}), the capital of a venture capital investment (see e.g. \cite{Bayraktar08}), the capital of annuity business where payments are continuously issued to annuitants while emerging profits arise in case annuitants die (see e.g. \cite{Ng09}), or capital of pharmaceutical and petroleum companies (see e.g. \cite{Zhao17}). However, SPLPs have completely different path properties from SNLPs. It is required to develop a distinct way based on {\color{black}excursion theory of Markov processes} to handle our auxiliary singular control problem (c.f. Section \ref{sec:AuSingCP}). For example, in addition to the fluctuation theory of SPLPs, we need resort to the excursion-theoretical approach, which is not necessary in the analysis of \cite{Mata22}, to derive the expression for the value function of a double barrier dividend and capital injection strategy.
In addition, we have to separately consider the two cases whether or not the spectrally positive L\'evy process has bounded path variation, while \cite{Mata22} does not need to treat these two cases separately. Actually, arguments involving fluctuation theory and excursion-theoretical approach are used for the case of bounded variation, approximation and limiting arguments are further needed to handle the case of unbounded variation.
We also note that, the dividend barriers associated with the optimal strategy of the auxiliary control problem and the target control problem are strictly positive, while the counterparts of \cite{Mata22} can be $0$ in the bounded path variation case.

Motivated by the work of \cite{Mata22},  in this paper, we raise a natural conjecture that the form of optimal strategy solving the Poisson observation version of De Finetti's dividend and capital injection problem remains a regime-modulated double barrier strategy when the uncontrolled spectrally negative Markov additive process of \cite{Mata22} is replaced with a spectrally positive Markov additive process, and aim to verify the conjecture. We first address an auxiliary optimal dividend and capital injection problem with a final payoff under a single spectrally positive L\'evy process, of which the optimal strategy is guessed to be some double barrier strategy. To confirm our guess, we need to derive the expression of the value function of a double barrier strategy for the auxiliary problem. This boils down to the derivation of solutions to the expected discounted total dividend payment minus the expected discounted total capital injection as well as the potential measure associated with the spectrally positive L\'evy process controlled by a double barrier dividend and capital injection strategy. During the derivation process, the cases that the underlying single spectrally positive L\'evy process is of bounded or unbounded variation should be considered separately, and the fluctuation theory and excursion-theoretical approach play important roles.
With the expression of the value function of a double barrier strategy in hand, we can characterize the optimal strategy of the auxiliary problem among the set of double barrier strategies, to which end we manage to express the derivative of the value function in terms of the Laplace transform and potential measure of the spectrally positive L\'evy process controlled by a single barrier  Poissonian dividend strategy until the first time it down-crosses $0$. This new compact expression facilitates us to identify the candidate optimal strategy among the set of double barrier dividend and capital injection strategies and the slope conditions of the value function of this candidate optimal double barrier strategy.
Using these slope conditions and the Hamilton-Jacobi-Bellman inequality approach of the control theory, the above candidate optimal double barrier strategy is verified as the optimal strategy for the auxiliary problem.
Finally, thanks to the results for the auxiliary control problem and a fixed point argument
for recursive iterations induced by the dynamic programming principle, the optimality of a regime-modulated double barrier  Poissonian dividend and capital injection strategy is proved for our target control problem.

{\color{black}The remainder of this paper is organized as follows. In Section~\ref{sec:model}, we formulate  a bail-out dividend control problem in the MAP framework and provide the main reuslt of the paper. In Section~\ref{sec:AuSingCP}, we introduce an auxiliary optimal dividend and capital injection problem with a terminal payoff function. Section~\ref{sec:proofthm2.1} is devoted to the proof of the main result.  Some proofs
and additional auxiliary results are respectively delegated to Appendix~\ref{sec:AppendixA}, Appendix~\ref{sec:AppendixB} and Appendix~\ref{app:C}. }

\section{Problem Formulation and Main Result}\label{sec:model}

In this section, we introduce the surplus process based on a spectrally positive Markov additive process (MAP) for a company and formulate a bail-out dividend control problem in this MAP framework, where capitals are allowed to be injected into the surplus process of the company so as to avoid its bankruptcy.

Let $(X,Y)=(X_t,Y_t)_{t\geq0}$ be a spectrally positive MAP defined on a filtered probability space $(\Omega,\F,\Fx,\Px)$ with the filtration $\Fx=(\F_t)_{t\geq0}$ satisfying the usual conditions, where, $X=(X_{t})_{t\geq 0}$ models the surplus process (before dividends are paid and capitals are injected) of a company, and, $Y=(Y_t)_{t\geq0}$ captures the changes in market behavior due to macroeconomic
transitions or macroeconomic readjustments, such as technological development, epidemics, and geopolitical issues (see, \cite{Mata22}, for example).  We assume that the process $Y$ is a continuous-time Markov chain with finite state space $I$ and generator $(\lambda_{ij})_{i,j\in I}$. {\color{blue}Conditionally on that $Y$ is in state $i\in I$, the process $X$ evolves as a spectrally positive L\'evy process $X^{i}=(X_t^{i})_{t\geq0}$ {(i.e., $X^{i}$ is a L\'evy process with upward jumps but no downward jumps)} until $Y$ switches to another state $I\ni j\neq i$, at which instant there is a downward jump in $X$.
More explicitly, if we denote by $(\ell_{n})_{n\geq 1}$ the times when $Y$ switches its states, i.e.
\begin{align}
    \ell_{0}:=0\quad \text{and}\quad  \ell_{n}:=\inf\{t>\ell_{n-1}; Y_{t}\neq Y_{t-}\},\quad n\geq 1,\nonumber
\end{align}
then the dynamics of $X$ is governed by
\begin{align}
\label{another def. X}
    X_t=
    &X_0
    +\sum_{n=0}^{\infty}\sum_{i\in I}\sum_{j\in I}\left(\left(X^{i}_{\ell_{n+1}-}-X^{i}_{\ell_{n}}+J_{ij}^{(n+1)}\right)\mathbf{1}_{\{\ell_{n+1}\leq t\}}
      +\left(X^{i}_{t}-X^{i}_{\ell_{n}}\right)\mathbf{1}_{\{\ell_{n}\leq t< \ell_{n+1}\}}
      \right)
    \nonumber
    \\
    &
    \times \mathbf{1}_{\{Y_{\ell_n}=i\}}\mathbf{1}_{\{Y_{\ell_{n+1}}=j\}}, \quad t\geq 0,
\end{align}
where $\{J_{ij}^{(n)}; i, j\in I\}_{n\geq 1}$ are a sequence of independent copies of $\{J_{ij}; i, j\in I\}$ with the latter being a finite number of pre-specified generic random variables, and
\begin{align}
    X^{i}_{t} = c_i t + \sigma_{i} B_t^{i} +\int_{0}^{t}\int_{(0,1)}x\overline{N^{i}}(ds,dx)+\int_{0}^{t}\int_{[1,\infty)}xN^{i}(ds,dx), \quad t\geq 0,\,\, i\in I,\nonumber
\end{align}
where $B_t^{i}$ is a standard Brownian motion, $N^{i}(ds, dx)$ (independent of $B_t^{i}$) is a Poisson random
measure on $[0,\infty) \times (0,\infty)$ with intensity measure $ds\nu_{i}(dx)$, and $\overline{N^{i}} (ds, dx) = N^{i} (ds, dx) -
ds\nu_{i} (dx)$ denotes the compensated random measure. It needs mentioning that, because $Y$ has a finite state space, the sum in \eqref{another def. X} is a finite sum almost surely for each finite $t$.
For each $i\in I$, let the Laplace exponent of the L\'evy process $X^i$, denoted by $\psi_i:[0,\infty)\rightarrow\mathbb{R}$, be given by the L\'evy–Khintchine formula
\begin{align}
\psi_i(\theta):=c_i \theta+\frac{\sigma_i^2}{2}\theta^2+\int_{(0,\infty)}(e^{-\theta z}-1+\theta z\mathbf{1}_{\{z<1\}})\nu_i(dz), \quad \forall \,\theta\geq 0,\,\, i\in I,
\end{align}
where $c_i\in\mathbb{R}$, $\sigma_i\geq 0$, and $\nu_i$ is the L\'{e}vy measure of $X^i$ with support $(0,\infty)$ such that $\int_{(0,\infty)}(1\wedge z^2)\nu_i(dz)<\infty$. We assumed throughout the paper that
\begin{align}\label{eq:AssEX1Levy}
\Ex[X^i_1]=-\psi_i'(0+)<\infty,\quad \forall \,i\in I.
\end{align}
}We also assume that $(X^{i})_{i\in I}$, $Y$ and $(J_{ij}^{(n)})_{i,j\in I,n\geq 1}$ are mutually independent{; and, the distribution function of $J_{ij}$, denoted as $F_{ij}$, is supported on the negative half line and has finite mean.}



We consider a bail-out dividend control problem in the above MAP framework, where capitals are allowed to be injected into the surplus process of the company so that its final surplus process always keeps non-negative. To formulate this problem, let $D=(D_t)_{t\geq0}$ and $R=(R_{t})_{t\geq0}$ (with $D_{0-}=R_{0-}=0$) be two non-decreasing, right-continuous and $\Fx$-adapted processes, which represent the cumulative amount of dividends and injected capitals  respectively.

Assume that the dividend payments can only be made at the arrival epochs $(T_n)_{n\geq1}$ of a Poisson process $N:=(N_t)_{t\geq0}$ with intensity $\gamma>0$, {where $N$ is independent of the above mentioned random blocks $(X^{i})_{i\in I}$, $Y$, and, $(J_{ij})_{i,j\in I}$}. Contrary to the dividend payments, the capital injection can be made continuously in time. Hence, the surplus process of the agent by taking into account the dividends and capital injection is given by
\begin{align}\label{eq:surplusU}
U_t := X_t-D_t+R_t,\quad \forall t\geq0.
\end{align}
The value function of the  Poissonian dividend control problem with capital injection is defined by
\begin{align}\label{eq:valuefcn}
V(x,i):=\sup_{(D,R)\in{\cal U}}\Ex_{x,i}\left[\int_{0}^{\infty}e^{-\int_{0}^{t}r_{Y_s}ds} d\left( D_{t}
-\phi R_{t}\right)\right],\quad\forall (x,i)\in[0,\infty)\times I,
\end{align}
where the conditional expectation $\Ex_{x,i}[\cdot]:=\Ex[\cdot|X_0=x,Y_0=i]$, $r_i>0$ for $i\in I$, $r_{Y_{t}}$ is the discount factor depending on the state of the Markov chain $Y$, and $\phi>1$ represents the cost of per unit of capital injected. Here, the set ${\cal U}$ is the admissible control set which is the space of non-decreasing, c\`{a}dl\`{a}g and adapted processes $(D,R)=(D_t,R_t)_{t\geq0}$ (with $D_{0-}=R_{0-}=0$) such that $U_t\geq0$ and $D_t=\int_0^t\Delta D_s d N_s$ (with $\Delta D_t=D_t-D_{t-}$) for all $t\geq 0$, and
\begin{align}
\label{2.3.addnew}
   {\Ex_{x,i}\left[\int_{0}^{\infty}e^{-\int_{0}^{t}r_{Y_s} d s} d R_{t}\right]<\infty, \quad (x,i)\in[0,\infty)\times I.}
\end{align}

It is seen that most of the existing works concerning dividend problem assume that the board of the company will continually track the surplus process and make dividend decisions accordingly, thereby probably resulting in a continuous dividend payment stream. However, in practice, it is more reasonable that the board of the company
checks the balance of the company on a  Poissonian basis, and then decides whether to pay dividends to the shareholders or keep these cash reserves inside in order to protect the company from future losses. In this vein, rather than continuous dividend streams, the shareholders shall receive lump sums of dividend payments at discrete time points.
To implement the idea of only acting on a  Poissonian basis in
time, whereas maintaining part of the mathematical transparency and elegance of acting in a continuous-time style, we assume that $D_t=\int_0^t\Delta D_s d N_s$, that is
\begin{align}\label{eq:Drep}
D_t=\int_0^{t}\Delta D_sdN_s=\sum_{n=1}^{\infty}\Delta D_{T_n}\mathbf{1}_{\{T_n\leq t\}},\quad \forall t\geq0.
\end{align}
Here ${\bf1}_A$ denotes the indicator r.v. w.r.t. an event $A\in{\cal F}$.
In addition, to guarantee the well-posedness of problem \eqref{eq:valuefcn}, we need to impose the condition \eqref{2.3.addnew}.

Note that problem \eqref{eq:valuefcn} is a singular stochastic control problem with state process given by the spectrally positive Markov additive process $(X,Y)=(X_t,Y_t)_{t\geq0}$. The difficulty for solving this singular stochastic control problem lies in (i) verifying that $x\to \widehat{V}(x,i)$ (c.f., \eqref{eq:valuefcn} and \eqref{eq:hatf00}) is a payoff function for all $i\in I$ (c.f. the proof of Theorem \ref{thm2.1}), where, {$\widehat{V}(\cdot,i)$ is called a payoff function if it is a continuous and concave function which satisfies $\widehat{V}^{\prime}_{+}(0+,i)\leq \phi$ and $\widehat{V}^{\prime}_{+}(\infty,i)\in [0,1]$ with $\widehat{V}^{\prime}_{+}(x,i)$ being the right derivative of $\widehat{V}(\cdot,i)$ at $x$}; (ii) characterizing the expression for the performance function of the double barrier  Poissonian-continuous-reflection dividend and capital injection strategy of the auxiliary problem \eqref{eq:auxiliarycontrol} (c.f., Proposition \ref{V.x}, however, the standard method such as fluctuation theory is not enough, we have to apply also the excursion-theoretical approach, which, as far as the authors know, has seldom (if not never) been used in dividend optimation problems); and (iii) identifying the slope conditions of the candidate optimal performance function of \eqref{eq:JPhixpib} with $b=b_{\Psi}$ (c.f., Lemma \ref{V.bw.parl}). The roadmap for solving \eqref{eq:valuefcn} can be depicted as follows:
\begin{itemize}
    \item[(i)] Solve the auxiliary control problem \eqref{eq:auxiliarycontrol} below with an exponential time horizon, a terminal payoff function $\Psi(x)$, and a SPLP as the underlying state process.

    \item[(ii)] Prove that, for each $i\in I$, the function
    $x\mapsto\widehat{V}(x,i)$ defined by \eqref{eq:hatf00} with the value function $V(x,i)$ given by \eqref{eq:valuefcn} is a payoff function, that is, $x\to \widehat{V}(x,i)$ is a continuous and concave function with $\widehat{V}^{\prime}_{+}(0+,i)\leq \phi$ and $\widehat{V}^{\prime}_{+}(\infty,i)\in[0,1]$ (c.f., Lemma \ref{lem.V.n}). 

\item[(iii)] Let $\tau_{i}$ be the first time when the Markov chain $Y$ switches its regime state from $Y_0=i$ to other states. Consider using the dynamic program {\color{black}(c.f., Proposition 3.1 of \cite{Mata22})} described as follows, for $(x,i)\in[0,\infty)\times I$,
\begin{align}\label{eq:dpp}
V(x,i)&=\sup_{(D,R)\in{\cal U}}\Ex_{x,i}\left[ \int_{0}^{\tau_{i}}e^{-\int_{0}^{t}r _{Y_s}ds}d(D_{t}-\phi R_{t})+e^{-\int_0^{\tau_{i}}r _{Y_s}ds}V(U_{\tau_{i}-}, Y_{\tau_{i}})\right]
\nonumber\\
&=\sup_{(D,R)\in{\cal U}}\Ex_{x,i}\left[ \int_{0}^{\tau_{i}}e^{-r_{i}t}d(D_{t}-\phi R_{t})+e^{-r_{i}\tau_{i}}\widehat{V}(U_{\tau_{i}-}, i)\right].
\end{align}
It is seen that the prime problem \eqref{eq:valuefcn} is equivalent to the auxiliary problem \eqref{eq:auxiliarycontrol} with the payoff function $\widehat{V}(x,i)$. Thus, once the auxiliary control problem is solved, the prime problem \eqref{eq:valuefcn} is then solved.


\end{itemize}

{\color{blue}Notice that we can prove that the value function $V(x,i)$ is concave with respect to the first argument by a method similar to that of Lemma 1 of \cite{Kulenko08}. This implies that the optimal dividend strategy is probably a barrier strategy.} For any vector $\overrightarrow{b}=(b_{i})_{i\in I}\in\mathbb{R}_{+}^{I}$, we define a  Poissonian dividend and capital injection strategy, denoted by  $(D_t^{0,\overrightarrow{b}},R_t^{0,\overrightarrow{b}})$, as follows:
\begin{eqnarray}
\label{eq:DbRb}
D_t^{0,\overrightarrow{b}} := \sum_{n=1}^{\infty}((U_{T_{n}-}^{0,\overrightarrow{b}}+\Delta X_{T_{n}}-b_{Y_{T_{n}}})\vee0)\mathbf{1}_{\{T_{n}\leq t\}},\quad
R_t^{0,\overrightarrow{b}}:=-\inf_{s\leq t}((X_s-D_s^{0,\overrightarrow{b}})\wedge0),
\end{eqnarray}
with $U_{t}^{0,\overrightarrow{b}}:=X_{t}-D_t^{0,\overrightarrow{b}}+R_t^{0,\overrightarrow{b}}$ and $\Delta X_{t}:=X_{t}-X_{t-}$. The controlled surplus process $(U_{t}^{0,\overrightarrow{b}})_{t\geq0}$ can be understood as the spectrally positive Markov additive process $X$ being reflected at the Poisson arrival times $(T_{n})_{n\geq 1}$ from above at the dynamic level $b_{Y_{t}}$ at time $t$ and being reflected continuously from below at the level $0$. We hence call $(D_t^{0,\overrightarrow{b}},R_t^{0,\overrightarrow{b}})_{t\geq 0}$ a {\color{black}strategy pair of}  Poissonian-continuous-reflection dividend and capital injection with dynamic dividend barrier $b_{Y_{t}}$ (at time $t$) and constant capital injection barrier $0$. The performance function of the strategy $(D_t^{0,\overrightarrow{b}},R_t^{0,\overrightarrow{b}})_{t\geq 0}$, denoted by $V_{0,\overrightarrow{b}}(x,i)$, is defined as, for $(x,i)\in\R_+\times I$,
\begin{align}\label{eq:valuefcn...}
V_{0,\overrightarrow{b}}(x,i):=\Ex_{x,i}\left[\int_{0}^{\infty}e^{-\int_{0}^{t}r_{Y_s}ds} d\left( D_t^{0,\overrightarrow{b}}
-\phi R_t^{0,\overrightarrow{b}}\right)\right].
\end{align}

The main result of the current paper is stated in the following Theorem \ref{thm2.1}. It confirms the optimality of some regime-modulated ( Poissonian-continuous-reflection) dividend and capital injection strategy for the prime problem \eqref{eq:valuefcn}. The proof of Theorem \ref{thm2.1} is deferred to Section 4.

\begin{theorem}\label{thm2.1}
There exists a vector $\overrightarrow{b^{*}}
\in\mathbb{R}_{+}^{I}$ such that $(D_t^{0,\overrightarrow{b^{*}}},R_t^{0,\overrightarrow{b^{*}}})_{t\geq 0}$ (which is defined by \eqref{eq:DbRb} with $\overrightarrow{b}$ replaced as $\overrightarrow{b^{*}}$) is {\color{black}an} optimal strategy for the prime problem \eqref{eq:valuefcn}, i.e.
\begin{align}\label{eq:VobV12}
   V_{0,\overrightarrow{b^{*}}}(x,i)=V(x,i),\quad \forall  (x,i)\in\mathbb{R}_{+}\times I,
\end{align}
where $V_{0,\overrightarrow{b^{*}}}(x,i)$ is given by \eqref{eq:valuefcn...} but with $\overrightarrow{b}$ replaced as $\overrightarrow{b^{*}}$.
\end{theorem}

We will prove Theorem~\ref{thm2.1} according to the roadmap (i)-(iii) described as above.

\section{Auxiliary Singular Control Problem}\label{sec:AuSingCP}

To solve the (primal) singular stochastic control problem \eqref{eq:valuefcn}, we introduce an auxiliary optimal dividend and capital injection problem with a terminal payoff function. In particular,
in the auxiliary control problem, we suppose that, when there is no control, the surplus process evolves as a single SPLP $X=(X_t)_{t\geq0}$ (see Appendix \ref{sec:AppendixA}).

Let $r>0$ be a discount factor, and, $\tau$ be an exponentially distributed random variable (r.v.) with mean $1/\lambda>0$. 
{We assume that $\tau$ is independent of the surplus process $X$ and the Poisson process $N$}. Let $\Psi:\R_+\to\R_+$ be a payoff function, i.e., it is a continuous and concave function with
$\Psi^{\prime}_{+}(0+)\leq \phi$ and $\Psi^{\prime}_{+}(\infty)\in[0,1]$, where $\Psi_+^{\prime}(x)$ represents the right derivative of $\Psi$ at $x$. In term of \eqref{eq:dpp}, the auxiliary optimal dividend and capital injection problem with the terminal payoff function $\Psi$, is described as
\begin{align}\label{eq:auxiliarycontrol}
 V_{\Psi}(x) := \sup_{\pi\in\Pi} J_{\Psi}(x;\pi)=\sup_{\pi\in\Pi}\Ex_x\left[\int_0^{\tau}e^{-r t}d(D_{t}-\phi R_t)+e^{-r \tau}\Psi(U_{\tau}^{\pi})\right],\quad x\geq0,
\end{align}
where $\mathbb{E}_x$ is the expectation under the law of $X$ given $X_0=x$ (c.f., Appendix \ref{sec:AppendixA}), $\pi=(D,R)$, and the surplus process under $\pi$, denote by $U^{\pi}=(U_t^{\pi})_{t\geq0}$, is of form \eqref{eq:surplusU} but with $X$ being a SPLP. In other words, we use the superscript ``$\pi$" to highlight the dependence of the surplus process on $\pi\in\Pi$. The surplus process under $\pi$ is given as
\begin{align}\label{eq:Drep2}
U_t^{\pi}=X_t-\sum_{n=1}^{\infty}\Delta D_{T_n}\mathbf{1}_{\{T_n\leq t\}}+R_t,\quad \forall t\geq0.
\end{align}
In \eqref{eq:auxiliarycontrol}, the admissible control set $\Pi$ is the space of $\pi=(D,R)$ such that $(D,R)$ are non-decreasing, c\`{a}dl\`{a}g and adapted processes (with $D_{0-}=R_{0-}=0$), $U_t^{\pi}\geq0$ and $D_t=\int_0^t\Delta D_s d N_s$ (with $\Delta D_t=D_t-D_{t-}$) for all $t\geq 0$, and, {$\Ex_x\left[\int_{0}^{\infty}e^{-(r+\lambda)t}dR_{t}\right]<\infty$}. Using \eqref{eq:Drep2}, the auxiliary control problem \eqref{eq:auxiliarycontrol} can be equivalently written as
\begin{align}\label{eq:auxiliarycontrol2}
 V_{\Psi}(x) &= \sup_{\pi\in\Pi} \int_0^{\infty}\lambda e^{-\lambda t}\Ex_x\left[\int_0^{t}e^{-r s}\left(\sum_{n=1}^{\infty}\Delta D_{T_n}\mathbf{1}_{\{T_n\in ds\}}-\phi dR_s\right)+e^{-r t}\Psi(U_{t}^{\pi})\right]dt\nonumber\\
 &{\color{blue}=\sup_{\pi\in\Pi} \Ex_x\left[\int_0^{\infty}e^{-rs}\int_s^\infty\lambda e^{-\lambda t}dt\left(\sum_{n=1}^{\infty}\Delta D_{T_n}\mathbf{1}_{\{T_n\in ds\}}-\phi dR_s\right)+\int_0^\infty\lambda e^{-\lambda t}e^{-rt}\Psi(U_t^{\pi})dt\right]}
 \nonumber\\
  &{\color{blue}=\sup_{\pi\in\Pi} \Ex_x\left[\int_0^{\infty}e^{-q s}\left(\sum_{n=1}^{\infty}\Delta D_{T_n}\mathbf{1}_{\{T_n\in ds\}}-\phi dR_s\right)+\int_0^\infty\lambda e^{-q t}\Psi(U_t^{\pi})dt\right]}
  \nonumber\\
 &=\sup_{\pi\in\Pi} \Ex_x\left[\sum_{n=1}^{\infty}e^{-qT_n}\Delta D_{T_n}-\phi\int_0^{\infty} e^{-qt} dR_t+\int_0^{\infty}\lambda e^{-qt}\Psi(U_{t}^{\pi})dt\right]
\end{align}
with $q:=r+\lambda${\color{blue}, where the validity of the interchange of integration in the second equality is due to the Fubini's theorem.}


Similar to \cite{Mata22} and \cite{Wang22.c}, we shall guess and verify that {\color{blue}the} optimal dividend and capital injection strategy can be connected to the so-called double barrier strategies. To do it, we introduce the following non-decreasing and adapted processes
\begin{align}\label{eq:DRb}
D_t^{0,b} := \sum_{n=1}^{\infty}((U_{T_{n}-}^{0,b}+\Delta X_{T_{n}}-b)\vee0)\mathbf{1}_{\{T_{n}\leq t\}}\,\,\,\text{and}\,\,\,
R_t^{0,b}:=-\inf_{s\leq t}((X_s-D_s^{0,b})\wedge0),\,\,\, t\geq 0,
\end{align}
where $b>0$ is a dividend barrier, and, $(U_{t}^{0,b}:=X_t-D_{t}^{0,b}+R_{t}^{0,b})_{t\geq0}$ is the surplus process associated with the double barrier dividend and injection strategy $(D^{0,b},R^{0,b})=(D_t^{0,b},R_t^{0,b})_{t\geq0}$. In fact, the process $U_{t}^{0,b}$ is the SPLP $X$ reflected from above at the constant dividend barrier level $b$ at the Poisson arrival times $(T_{n})_{n\geq 1}$, and, reflected continuously from below at the capital injection level $0$. We hence call $(D_t^{0,b},R_t^{0,b})_{t\geq 0}$ a strategy pair of  Poissonian-continuous-reflection dividend and capital injection with constant dividend barrier $b$ and constant capital injection barrier $0$. Here, we use `` Poissonian" because, under $(D^{0,b},R^{0,b})$, dividends will be paid only at the Poisson arrival times $(T_{n})_{n\geq 1}$.

The performance function corresponding to this  double barrier dividend and capital injection strategy $\pi^b=(D^{0,b},R^{0,b})$ is given by (see \eqref{eq:auxiliarycontrol} and \eqref{eq:auxiliarycontrol2})
\begin{align}\label{eq:JPhixpib}
    J_{\Psi}(x;\pi^b)=\Ex_x\left[\int_0^{\infty}e^{-qt}d(D^{0,b}_t-\phi R_t^{0,b})+\lambda\int_0^{\infty}e^{-qt}\Psi(U^{0,b}_{t})dt\right].
\end{align}

We next derive an expression for the performance function $J_{\Psi}(x;\pi^b)$ defined by \eqref{eq:JPhixpib}, which is given by the following proposition.

\begin{proposition}\label{V.x}
It holds that
\begin{align}\label{eq:JPhipib-exp}
&J_{\Psi}(x;\pi^b)=-\frac{\gamma\left[\overline{Z}_q(b-x)+\frac{\psi^{\prime}(0+)}{q}\right]}{q+\gamma}
+\frac{\left(\gamma Z_q(b)-\phi(q+\gamma)\right)\left[Z_q(b-x,\Phi_{q+\gamma})+\frac{\gamma}{q}Z_q(b-x)\right]}{(q+\gamma)\Phi_{q+\gamma}Z_q(b,\Phi_{q+\gamma})}\nonumber\\
&\quad+\frac{\lambda\Psi(b)}{q}Z_q(b-x)-\lambda\int_0^b\Psi(y)W_q(y-x)dy-\frac{qZ_q(b-x,\Phi_{q+\gamma})+\gamma Z_q(b-x)}{q\Phi_{q+\gamma}Z_q(b,\Phi_{q+\gamma})}\\
&\quad\times\Bigg\{\lambda\int_0^b\Psi^{'}_+(y)W_q(y)dy+\lambda\int_0^{\infty}\Psi^{'}_+(b+y)\left[W_q(b+y)+\gamma\int_0^yW_q(b+y-z)W_{q+\gamma}(z)dz\right.\nonumber\\
&\quad
\left.-\frac{\Phi_{q+\gamma}Z_q(b,\Phi_{q+\gamma})\left[Z_q(b-x+y)
+\gamma\int_0^yZ_q(b-x+y-z)W_{q+\gamma}(z)dz\right]}{qZ_q(b-x,\Phi_{q+\gamma})+\gamma Z_q(b-x)}
\right]dy\Bigg\},\quad x>0,\nonumber
\end{align}
while $J_{\Psi}(x;\pi^b)=\phi x+J_{\Psi}(0;\pi^b)$ for all $x\leq0$. Here, $\psi(\theta)$ is defined by \eqref{eq:psiLevy}, $Z_q(x)$ (resp. $\overline{Z}_q(x)$) is defined by \eqref{eq:ZqLevy} (resp. \eqref{eq:anti-deriZq}), $\Phi_q$ is given in \eqref{eq:LTWq}, $Z_{q}(x,\Phi_{q+\lambda})$ is given by \eqref{eq:Zqxs}, and $W_q(x)$ is the $q$-scale function of the SPLP $X$ in \eqref{eq:LTWq}.
\end{proposition}

The proof of Proposition~\ref{V.x} is tedious, and we report the proof in Appendix~\ref{sec:AppendixB}. Using the explicit expression \eqref{eq:JPhipib-exp} for the performance function $J_{\Psi}(x;\pi^b)$ in Proposition~\ref{V.x}, the following lemma gives the smoothness of $\R\ni x\to J_{\Psi}(x;\pi^b)$.
\begin{lemma}\label{lem.par.b}
For any $b>0$, $J_{\Psi}(x;\pi^b)$ is {continuous on $\R$, continuously differentiable on $\R\setminus \{0\}$}. Moreover, if the SPLP $X$ has paths of unbounded variation, then $J_{\Psi}(x;\pi^b)$ is {twice continuously differentiable on $\R\setminus\{0,b\}$}. 
\end{lemma}
\begin{proof}
{Recall that (i)  $\overline{Z}_{q}$ and $Z_{q}$ are continuous on $\R$ and continuously differentiable on $\R\setminus \{0\}$; (ii) $W_{q}$ is continuous, right and left differentiable on $\R\setminus\{0\}$, differentiable on $\R\setminus\{0\}$ except for countably many points; (iii) $W_{q}$ only appears in the integrands of the expression of $J_{\Psi}(x;\pi^b)$; and
\begin{align}
    \text{(iv)}\quad \left[\int_0^b\Psi(y)W_q(y-x)dy\right]^{\prime}=&\left[\int_x^b\Psi(y)W_q(y-x)dy\right]^{\prime}
    =-
    \mathbf{1}_{\{0<x<b\}}\int_x^b\Psi(y)W_q^{\prime}(y-x)dy
\nonumber\\
=&
-\mathbf{1}_{\{0<x<b\}}\int_0^{b-x}\Psi(x+y)W_q^{\prime}(y)dy,\nonumber
\end{align}
with $\Psi$ being concave on $\R_{+}$ (hence, differentiable on $\R_{+}$ except for countably many points).
Hence, $J_{\Psi}(x;\pi^b)$ is continuous on $\R$ and continuously differentiable on $\R\setminus \{0\}$.

Furthermore, if the process $X$ has paths of unbounded variation, then (i)  $W_{q}$ is continuous on $\R$, continuously differentiable on $\R\setminus\{0\}$; and, (ii) $\overline{Z}_{q}$ and $Z_{q}$ are twice continuously differentiable on $\R\setminus \{0\}$.
 Therefore, $J_{\Psi}(x;\pi^b)$ is twice continuously differentiable on $\R\setminus\{0,b\}$.
}
\end{proof}

Bearing in mind that the optimal strategy solving the auxiliary problem \eqref{eq:auxiliarycontrol2} is conjectured to be some double barrier ( Poissonian-continuous-reflection) dividend and capital injection strategy, we hence need to characterize the barrier levels corresponding to the optimal strategy among the set of double barrier  Poissonian-continuous-reflection strategies. To achieve this goal, we follow the spirits of \cite{Mata22} and \cite{Wang22.c} to express the derivative of the performance function $J_{\Psi}(x;\pi^b)$ at $b$ (i.e., $\left.\frac{d}{dx}J_{\Psi}(x;\pi^b)\right|_{x=b}$) in terms of the Laplace transform and potential measure of the SPLP controlled by a single barrier  Poissonian-reflection dividend strategy until the first time it down-crosses zero.
This new compact expression of $\left.\frac{d}{dx}J_{\Psi}(x;\pi^b)\right|_{x=b}$ will facilitate us to identify the dividend barrier of the candidate optimal strategy among the set of double barrier  Poissonian-continuous-reflection dividend and capital injection strategies. To implement these ideas, we first provide some preliminary results concerning the SPLP controlled by a single barrier  Poissonian dividend strategy. That is to consider
\begin{align}\label{eq:single-Utb}
\begin{cases}
\displaystyle   U_t^b=X_t-D_t^b,\quad t\geq0,\\[0.6em]
\displaystyle   D_t^b=\sum_{n=1}^{\infty}\left((U_{T_n-}^b+\Delta X_{T_n}-b)\vee 0\right)\mathbf{1}_{\{T_n\leq t\}},\quad t\geq0.
\end{cases}
\end{align}
We also introduce the following down-crossing and up-crossing times of $U_t^b$ as
\begin{align}\label{eq:kappaa-+}
 \overline{\kappa}_a^-:=\inf\{t\geq0;~U_t^b<a\},\quad \overline{\kappa}_c^+:=\inf\{t\geq0;~U_t^b>c\},\quad \forall a,c\in\R.
\end{align}

Then, we can give, in the following proposition, an alternative expression for the derivative at $b$ of the performance function associated with the double barrier dividend and capital injection strategy (with dividend barrier $b$ and capital injection barrier $0$) in terms of the Laplace transform of $\overline{\kappa}_0^-$ and the potential measure of $U_t^b$ until $\overline{\kappa}_0^-$. {We emphasize once more that Proposition \ref{LTkappa0-} is crucial in identifying the dividend barrier (to be defined in \eqref{eq:bpsi}) corresponding to the
optimal strategy among the set of double barrier strategies with dividend barriers $b\geq 0$ and capital injection barrier $0$. What is more, the double barrier strategy with dividend barrier given by \eqref{eq:bpsi} and capital injection barrier $0$ will finally be proved to be the optimal strategy for the control problem \eqref{eq:auxiliarycontrol} (see, Theorem \ref{them.3.1}).}
\begin{proposition}\label{LTkappa0-}
We have
\begin{align*}
\frac{d}{dx}J_{\Psi}(x;\pi^b)|_{x=b}=\frac{\phi-1-{\Ex}_b\left[\int_0^{\overline{\kappa}_0^-}e^{-qt}(q\phi-\lambda\Psi_+^{\prime}(U_t^b))dt\right]}{\Ex_b\left[e^{-q\overline{\kappa}_0^-}\right]Z_q(b,\Phi_{q+\gamma})}+1,\quad b>0,
\end{align*}
where, $Z_q(x,s)$ is defined by \eqref{eq:Zqxs} in Appendix~\ref{sec:AppendixA}.
\end{proposition}

To prove Proposition~\ref{LTkappa0-} as well as the upcoming Lemma \ref{V.bw.parl}, we need the following two auxiliary lemmas. The first lemma computes the potential measure of the process $U^b=(U_t^b)_{t\geq0}$ until the first time it down-crosses $0$ and its proof is reported in Appendix~\ref{app:C}.
\begin{lemma}\label{lem.pm}
We have
\begin{align}\label{U.b}
&\Px_x(U_{e_q}^{b}\in dy;e_q<\overline{\kappa}_0^-)=
q\bigg[Z_q(b-x,\Phi_{q+\gamma})W_{q+\gamma}(y-b)-W_{q+\gamma}(y-x)\nonumber\\
&\qquad+\gamma\int_0^{b-x}W_q(b-x-z)W_{q+\gamma}(y-b+z) d z\bigg]\mathbf{1}_{(b,\infty)}(y)dy-qW_q(y-x)\mathbf{1}_{(0,b)}(y)dy\nonumber\\
&\qquad+\frac{\gamma Z_q(b-x)+qZ_q(b-x,\Phi_{q+\gamma})}{\gamma Z_q(b)+qZ_q(b,\Phi_{q+\gamma})}\bigg[qW_q(y)\mathbf{1}_{(0,b)}(y) d y+q\bigg(W_{q+\gamma}(y)\nonumber\\
&\qquad-Z_q(b,\Phi_{q+\gamma})W_{q+\gamma}(y-b)-\gamma\int_0^bW_q(b-z)W_{q+\gamma}(y-b+z)dz\bigg)\mathbf{1}_{(b,\infty)}(y)dy\bigg],\,\,x>0.
\end{align}
\end{lemma}
The second lemma gives the Laplace transform of the first passage time of $0$ for the SPLP with dividends subtracted according to the single barrier  Poissonian dividend strategy with barrier level $b$. {A detailed proof of this lemma can be found in Corollary 3.1 (ii) of \cite{Avram18}.}
\begin{lemma}\label{lem.stop}
We have
\begin{align}\label{over.kappa}
\Ex_x\left[e^{-q\overline{\kappa}_0^-}\right]=\frac{\gamma Z_q(b-x)+qZ_q(b-x,\Phi_{q+\gamma})}{\gamma Z_q(b)+qZ_q(b,\Phi_{q+\gamma})},\quad x>0,
\end{align}
where, $\overline{\kappa}_0^-$ is defined in \eqref{eq:kappaa-+} with $a=0$.
\end{lemma}

With the preparations made in Lemmas \ref{lem.pm} and \ref{lem.stop}, we are in the position to prove Proposition~\ref{LTkappa0-}.
\begin{proof}[Proof of Proposition~\ref{LTkappa0-}]
By applying Lemma \ref{lem.pm}, Lemma \ref{lem.stop} and \eqref{3.17}, we arrive at
\begin{align}
\frac{d}{dx}J_{\Psi}(x;\pi^b)|_{x=b}
&=\frac{\gamma}{q+\gamma}-\frac{\gamma Z_q(b)-\phi(q+\gamma)}{(q+\gamma)Z_q(b,\Phi_{q+\gamma})}
\nonumber\\
&\quad+\frac{1}{Z_q(b,\Phi_{q+\gamma})}
\Bigg\{\lambda\int_{0+}^bW_q(y)\Psi_+^{\prime}(y)dy+\lambda\int_0^{\infty}\Psi_+^{\prime}(b+y)\bigg[W_q(b+y)\nonumber\\
&\quad+\gamma\int_0^yW_q(b+y-z)W_{q+\gamma}(z)dz-{Z_q(b,\Phi_{q+\gamma})}W_{q+\gamma}(y)\bigg]dy\Bigg\}\label{continuityofv'}
\\
&=
\frac{\lambda\Ex_b\left[\int_0^{\overline{\kappa}_0^-}e^{-qt}\Psi_{+}^{\prime}(U_t^b)dt\right]+\phi\Ex_b\left[e^{-q\overline{\kappa}_0^-}\right]-1}{\Ex_b\left[e^{-q\overline{\kappa}_0^-}\right]Z_q(b,\Phi_{q+\gamma})}+1
\nonumber\\
&=\frac{\phi-1-\Ex_b\left[\int_0^{\overline{\kappa}_0^-}e^{-qt}(q\phi-\lambda\Psi_+^{\prime}(U_t^b))dt\right]}{\Ex_b\left[e^{-q\overline{\kappa}_0^-}\right]Z_q(b,\Phi_{q+\gamma})}+1,\nonumber
\end{align}
which is exactly the desired result.
\end{proof}

{We mention that both \cite{Mata22} and \cite{Wang22.c} proved that their optimal strategy were of double barrier type; what is more, the dividend barrier corresponding to their optimal strategy, turned out to be the smallest $b\geq0$ such that the derivative at $b$ of the value function of the double barrier strategy with dividend barrier $b$ and capital injection barrier $0$ is less that 1.
Motivated by \cite{Mata22} and \cite{Wang22.c}, we conjecture, with the help of Proposition~\ref{LTkappa0-}, that the optimal  Poissonian dividend and capital injection strategy of the control problem \eqref{eq:auxiliarycontrol} (or, equivalently, \eqref{eq:auxiliarycontrol2}) is a double barrier strategy with capital injection barrier $0$ and dividend barrier $b_{\Psi}$ given by}
\begin{align}\label{eq:bpsi}
b_{\Psi}&:=\inf\left\{b\geq0;~\frac{d}{dx}J_{\Psi}(x;\pi^b)|_{x=b}\leq 1\right\},
\end{align}
where $\inf\emptyset=+\infty$ by convention. 
{This conjecture will be verified in Theorem \ref{them.3.1} below. To prepare for the proof of Theorem \ref{them.3.1}, we need the following Lemmas \ref{lem.b.w}-\ref{max.V} as important intermediate steps.
We begin with}
\begin{lemma}\label{lem.b.w}
Let $b_{\Psi}$ be defined by \eqref{eq:bpsi}. Then $b_{\Psi}\in(0,\infty)$.
\end{lemma}

\begin{proof}Thanks to Proposition~\ref{LTkappa0-}, we may rewrite $b_{\Psi}$ as
\begin{align*}
b_{\Psi}=\inf\left\{b\geq0;~\ \phi-1-\mathbb{E}_b\left[\int_0^{\overline{\kappa}_0^-}e^{-qt}(q\phi-\lambda\Psi_+^{\prime}(U_t^b))dt\right]\leq 0\right\}.
\end{align*}
Since the terminal payoff function $\Psi$ is continuous and concave over $[0,\infty)$ with
$\Psi^{\prime}_{+}(0+)\leq \phi$ and $\Psi^{\prime}_{+}(\infty)\in[0,1]$, we have $b\mapsto q\phi-\lambda\Psi_+^{\prime}(b)$ is non-decreasing. On the other hand, by definition, we know that both $\overline{\kappa}_0^-$ and the process $U^b=(U_t^b)_{t\geq0}$ are non-decreasing with respect to $b$, which combined with the concavity of $\Psi$ results in that the mapping
\begin{align}\label{3.49}
b\mapsto\phi-1-\Ex_b\left[\int_0^{\overline{\kappa}_0^-}e^{-qt}(q\phi-\lambda\Psi_+^{\prime}(U_t^b))dt\right]~\text{is non-increasing.}
\end{align}
Moreover, using the spatial homogeneity of L\'evy processes and the dominated convergence theorem, it can be verified that
\begin{align}\label{3.50}
&\lim_{b\rightarrow\infty}\left\{\phi-1-\Ex_b\left[\int_0^{\overline{\kappa}_0^-}e^{-qt}(q\phi-\lambda\Psi_+^{\prime}(U_t^b))dt\right]\right\}\nonumber\\
&\qquad=
\phi-1-\lim_{b\rightarrow\infty}\Ex_0\left[\int_0^{\infty}e^{-qt}\mathbf{1}_{\{t\leq\overline{\kappa}_{-b}^-\}}(q\phi-\lambda\Psi_+^{\prime}(b+U_t^0))dt\right]\nonumber\\
&\qquad=\frac{\lambda}{q}\Psi_+^{\prime}(\infty)-1<0.
\end{align}
By letting $x=b$ and $b\downarrow0$ in Lemma \ref{lem.stop}, it follows that $\lim\limits_{b\downarrow 0}\overline{\kappa}_b^-=0$, a.s., which together with the dominated convergence theorem gives
\begin{align}\label{3.51}
\lim_{b\rightarrow0}\left\{\phi-1-\mathbb{E}_b\left[\int_0^{\overline{\kappa}_0^-}e^{-qt}(q\phi-\lambda\Psi_+^{\prime}(U_t^b)) d t\right]\right\}
=\phi-1>0.
\end{align}
Piecing together \eqref{3.50}, \eqref{3.51} and the non-increasing property of the mapping defined by \eqref{3.49} yields the desired result. The proof of Lemma \ref{lem.b.w} is complete.
\end{proof}

With the candidate optimal dividend barrier defined in \eqref{eq:bpsi}, we can now investigate, in the upcoming Lemma \ref{V.bw.parl}, the analytical properties, especially the slope conditions, of $x\to J_{\Psi}(x;\pi^{b_{\Psi}})$ of the double barrier ( Poissonian-continuous-reflection) dividend and capital injection strategy with dividend barrier $b_{\Psi}$ and capital injection barrier $0$.
\begin{lemma}\label{V.bw.parl}
The function $x\to J_{\Psi}(x;\pi^{b_{\Psi}})$ is increasing and concave over $\R$. Moreover, we have $\frac{d}{dx}J_{\Psi}(x;\pi^{b_{\Psi}})\in [1,\phi]$ for all $x\in(0,b_{\Psi})$; and, $\frac{d}{dx}J_{\Psi}(x;\pi^{b_{\Psi}})\in [0,1]$ for all $x\geq b_{\Psi}$.
\end{lemma}

\begin{proof}
Using Lemma \ref{lem.b.w}, it holds that
\begin{eqnarray}\label{v.bw.par.x}
\hspace{-0.3cm}&&\hspace{-0.3cm}\frac{d}{dx}J_{\Psi}(x;\pi^{b_{\Psi}})
=
\frac{\gamma}{q+\gamma}Z_q(b_{\Psi}-x)-\frac{\gamma Z_q(b_{\Psi})-\phi(q+\gamma)}{(q+\gamma)Z_q(b_{\Psi},\Phi_{q+\gamma})}Z_q(b_{\Psi}-x,\Phi_{q+\gamma})
\nonumber\\
\hspace{-0.3cm}&&\hspace{-0.3cm}\quad -\lambda\int_{0+}^{b_{\Psi}} W_q(y-x)\Psi^{\prime}_+(y)dy+\frac{Z_q(b_{\Psi}-x,\Phi_{q+\gamma})}{Z_q(b_{\Psi},\Phi_{q+\gamma})}\Bigg\{\lambda\int_{0+}^{b_{\Psi}}W_q(y)\Psi^{\prime}_+(y)dy
\nonumber\\
\hspace{-0.3cm}&&\hspace{-0.3cm}\quad +\lambda\int_0^{\infty}\Psi^{\prime}_+(b_{\Psi}+y)\Big[W_{q}(b_{\Psi}+y)+\gamma\int_0^yW_q(b_{\Psi}+y-z)W_{q+\gamma}(z)dz-\frac{Z_q(b_{\Psi},\Phi_{q+\gamma})}{Z_q(b_{\Psi}-x,\Phi_{q+\gamma})}\nonumber\\
\hspace{-0.3cm}&&\hspace{-0.3cm}\quad
\times(W_{q}(b_{\Psi}-x+y)+\gamma\int_0^yW_q(b_{\Psi}-x+y-z)W_{q+\gamma}(z)dz)\Big]dy
\Bigg\}
\nonumber\\
\hspace{-0.3cm}&=&\hspace{-0.3cm}
\frac{\gamma}{q+\gamma}Z_q(b_{\Psi}-x)
-\lambda\int_{0+}^{b_{\Psi}}W_q(y-x)\Psi^{\prime}_+(y) d y
\nonumber\\
\hspace{-0.3cm}&&\hspace{-0.3cm}
+\lambda\int_0^{\infty}\Psi^{\prime}_+(b_{\Psi}+y)\Big(W_{q+\gamma}(y)Z_q(b_{\Psi}-x,\Phi_{q+\gamma})-W_{q}(y-x+b_{\Psi})
\nonumber\\
\hspace{-0.3cm}&&\hspace{-0.3cm}
-\gamma\int_0^yW_q(y-z-x+b_{\Psi})W_{q+\gamma}(z) d z\Big) d y
+Z_q(b_{\Psi}-x,\Phi_{q+\gamma})
\bigg\{\underbrace{-\frac{\gamma Z_q(b_{\Psi})-\phi(q+\gamma)}{(q+\gamma)Z_q(b_{\Psi},\Phi_{q+\gamma})}}_{\raisebox{-0.5pt}{\textcircled{1}}}
\nonumber\\
\hspace{-0.3cm}&&\hspace{-0.3cm}
+\frac{1}{Z_q(b_{\Psi},\Phi_{q+\gamma})}\Big[\underbrace{\lambda\int_{0+}^{b_{\Psi}}W_q(y)\Psi^{\prime}_+(y) d y}_{\raisebox{-0.5pt}{\textcircled{2}}}
+\underbrace{\lambda\int_0^{\infty}\Psi^{\prime}_+(b_{\Psi}+y)\Big(W_{q}(b_{\Psi}+y)}_{\raisebox{-0.5pt}{\textcircled{3}} \text{ \,to be continued}}
\nonumber\\
\hspace{-0.3cm}&&\hspace{-0.3cm}
\underbrace{+\gamma\int_0^yW_q(b_{\Psi}+y-z)W_{q+\gamma}(z) d z-W_{q+\gamma}(y)Z_q(b_{\Psi},\Phi_{q+\gamma})\Big) d y}_{\raisebox{-0.5pt}{\textcircled{3}}\text{ continued}}\Big]\bigg\},\quad x>0.
\end{eqnarray}
By \eqref{continuityofv'}, one sees that $b\mapsto\frac{d}{dx}J_{\Psi}(x;\pi^b)|_{x=b}$ in continuous on $(0,\infty)$, which, together with the definition of $b_{\Psi}$, implies that $\frac{d}{dx}J_{\Psi}(x;\pi^{b_{\Psi}})|_{x=b_{\Psi}}=1$.
Therefore, by letting $x=b_{\Psi}$ in \eqref{continuityofv'} and using \eqref{3.17}, one gets
\begin{align}\label{V.bw.par.b}
\frac{q}{q+\gamma}&=\frac{d}{dx}J_{\Psi}(x;\pi^{b_{\Psi}})|_{x=b_{\Psi}}-\frac{\gamma}{q+\gamma}=
\raisebox{-0.5pt}{\textcircled{1}}
+\frac{1}{Z_q(b_{\Psi},\Phi_{q+\gamma})}\left(\raisebox{-0.5pt}{\textcircled{2}}+
\raisebox{-0.5pt}{\textcircled{3}}\right).
\end{align}
Equation \eqref{V.bw.par.b} can be equivalently written as
\begin{align}\label{V.bw.par.b.add}
\raisebox{-0.5pt}{\textcircled{2}}+
\raisebox{-0.5pt}{\textcircled{3}}=&
\left(\frac{q}{q+\gamma}+\frac{\gamma Z_q(b_{\Psi})-\phi(q+\gamma)}{(q+\gamma)Z_q(b_{\Psi},\Phi_{q+\gamma})}
\right)Z_q(b_{\Psi},\Phi_{q+\gamma})
\nonumber\\
=&\frac{\gamma Z_q(b_{\Psi})+q Z_q(b_{\Psi},\Phi_{q+\gamma})}{q+\gamma}-\phi.
\end{align}
Putting together \eqref{continuityofv'}, Lemma \ref{lem.stop}, \eqref{v.bw.par.x}, \eqref{V.bw.par.b.add}, and \eqref{3.17} yields
\begin{align*}
\frac{d}{dx}J_{\Psi}(x;\pi^{b_{\Psi}})
&=
\underbrace{\frac{\gamma Z_q(b_{\Psi}-x)+qZ_q(b_{\Psi}-x,\Phi_{q+\gamma})}{q+\gamma}}_{\raisebox{-0.5pt}{\textcircled{4}}}-\lambda\int_{0+}^{b_{\Psi}}W_q(y-x)\Psi_+^{\prime}(y)dy
\nonumber\\
&\quad
+\lambda\int_0^{\infty}\Psi^{\prime}_+(b_{\Psi}+y)\Big(W_{q+\gamma}(y)Z_q(b_{\Psi}-x,\Phi_{q+\gamma})-W_{q}(y-x+b_{\Psi})
\nonumber\\
&\quad-\gamma\int_0^yW_q(y-z-x+b_{\Psi})W_{q+\gamma}(z)dz\Big)dy
\nonumber\\
&=
\underbrace{\frac{\gamma Z_q(b-x)+qZ_q(b-x,\Phi_{q+\gamma})}{\gamma Z_q(b)+qZ_q(b,\Phi_{q+\gamma})}
\left(\frac{\gamma Z_q(b_{\Psi})+q Z_q(b_{\Psi},\Phi_{q+\gamma})}{q+\gamma}-\phi\right)}_{\raisebox{-0.5pt}{\textcircled{4}}\text{ to be continued}}
\nonumber\\
&\quad
\underbrace{+\phi \frac{\gamma Z_q(b-x)+qZ_q(b-x,\Phi_{q+\gamma})}{\gamma Z_q(b)+qZ_q(b,\Phi_{q+\gamma})}}_{\raisebox{-0.5pt}{\textcircled{4}}\text{ continued}}
-\lambda\int_{0+}^{b_{\Psi}}W_q(y-x)\Psi_+^{\prime}(y)dy
\nonumber\\
&\quad
+\lambda\int_0^{\infty}\Psi^{\prime}_+(b_{\Psi}+y)\Big(W_{q+\gamma}(y)Z_q(b_{\Psi}-x,\Phi_{q+\gamma})-W_{q}(y-x+b_{\Psi})
\nonumber\\
&\quad-\gamma\int_0^yW_q(y-z-x+b_{\Psi})W_{q+\gamma}(z)dz\Big)dy
\nonumber\\
&=
\phi\Ex_x\left[e^{-q\overline{\kappa}_0^-}\right]
+\frac{\gamma Z_q(b-x)+qZ_q(b-x,\Phi_{q+\gamma})}{\gamma Z_q(b)+qZ_q(b,\Phi_{q+\gamma})}
\left(\raisebox{-0.5pt}{\textcircled{2}}+
\raisebox{-0.5pt}{\textcircled{3}}\right)
\nonumber\\
&\quad -\lambda\int_{0+}^{b_{\Psi}}W_q(y-x)\Psi_+^{\prime}(y)dy
\nonumber\\
&\quad
+\lambda\int_0^{\infty}\Psi^{\prime}_+(b_{\Psi}+y)\Big(W_{q+\gamma}(y)Z_q(b_{\Psi}-x,\Phi_{q+\gamma})-W_{q}(y-x+b_{\Psi})
\nonumber\\
&\quad-\gamma\int_0^yW_q(y-z-x+b_{\Psi})W_{q+\gamma}(z)dz\Big)dy
\nonumber\\
&=
\phi\Ex_x\left[e^{-q\overline{\kappa}_0^-}\right]+\lambda\Ex_x\left[\int_0^{\overline{\kappa}_0^-}e^{-qt}\Psi^{\prime}_+(U_t^{b_{\Psi}})dt\right]
\nonumber\\
&=
\phi-\Ex_x\left[\int_0^{\overline{\kappa}_0^-}e^{-qt}\left(q\phi-\lambda\Psi^{\prime}_+(U_t^{b_{\Psi}})\right)dt\right].
\end{align*}
Recall that the terminal payoff function $\Psi$ is continuous and concave over $[0,\infty)$ with
$\Psi^{\prime}_{+}(0+)\leq \phi$ and $\Psi^{\prime}_{+}(\infty)\in[0,1]$. It follows that $\frac{d}{dx}J_{\Psi}(x;\pi^{b_{\Psi}})|_{x=0}\leq\phi$ and $x\to J_{\Psi}(x;\pi^{b_{\Psi}})$ is no-increasing on $\R$. Moreover, we derive from $\frac{d}{dx}J_{\Psi}(x;\pi^{b_{\Psi}})|_{x=b_{\Psi}}=1$ that $\frac{d}{dx}J_{\Psi}(x;\pi^{b_{\Psi}})\in[1,\phi]$ for all $x\in(0,b_{\Psi})$; while $\frac{d}{dx}J_{\Psi}(x;\pi^{b_{\Psi}})\in[0,1]$ for all $x\in[b_{\Psi},\infty)$. The proof is complete.
\end{proof}

The following lemma gives a verification lemma associated with the auxiliary problem \eqref{eq:auxiliarycontrol2}. Since its proof is standard, we omit it.
\begin{lemma}[Verification lemma]\label{lem.V.L}
Assume that there exists a non-decreasing and twice continuously differentiable  (if $\sigma\neq0$ in \eqref{eq:A}) or once continuously differentiable (if $\sigma=0$ in \eqref{eq:A}) function $\R\ni x\to v(x)$ satisfying the following variational inequality
\begin{align}\label{HJB}
\max\left\{\left(\mathcal{A}-q\right)v(x)+\lambda\Psi(x)+\gamma\sup_{z\in[0,x]}(z+v(x-z)-v(x)), v'(x)-\phi\right\}\leq 0.
\end{align}
Here, the operator $\mathcal{A}$ that acts on $C^2(\R)$ is defined by, for any $f\in C^2(\R)$,
\begin{align}\label{eq:A}
\mathcal{A}f(x):=\frac{\sigma^2}{2}f''(x)-cf^{\prime}(x)+\int_{(0,\infty)}\left(f(x+y)-f(x)-f^{\prime}(x)y\mathbf{1}_{(0,1)}(y)\right)\nu(dy),~\forall x\in\R,
\end{align}
where $(c,\sigma, \nu)$ is the L\'{e}vy triplet of $X$ (c.f., Appendix \ref{sec:AppendixA}).
Then, it holds that $J_{\Psi}(x;\pi)\leq v(x)$ for all $x\in\R$ and $\pi\in\Pi$, with the objective functional $J_{\Psi}(x;\pi)$ being defined by \eqref{eq:auxiliarycontrol}.
\end{lemma}

Recall the closed-form representation of $J_{\Psi}(x;\pi^b)$ provided in Proposition \ref{V.x}. We next aim to verify that the performance function $J_{\Psi}(x;\pi^{b_{\Psi}})$ indeed satisfies the variational inequality \eqref{HJB} {(If this is done, then, by the verification Lemma \ref{lem.V.L}, the double barrier dividend and capital injection strategy with dividend barrier $b_{\Psi}$ and capital injection barrier $0$ is indeed the optimal strategy for the auxiliary control problem \eqref{eq:auxiliarycontrol})}. Firstly, we obtain
\begin{lemma}
\label{3.9.add.new.l}
For any $b>0$, we have
\begin{align}\label{A.q.x}
\begin{cases}
\displaystyle ({\cal A}-q)J_{\Psi}(x;\pi^b)+\lambda \Psi(x)=0, & x\in(0,b),\\[0.6em]
\displaystyle ({\cal A}-q)J_{\Psi}(x;\pi^b)+\lambda\Psi(x)=-\gamma(x-b+J_{\Psi}(b;\pi^b)-J_{\Psi}(x;\pi^b)), & x\in[b,\infty).
 \end{cases}
\end{align}
\end{lemma}
We can apply the closed-form representation of $J_{\Psi}$ to prove \eqref{A.q.x}. Moreover, for the barrier $b_{\Psi}>0$ defined in \eqref{eq:bpsi}, we get from Lemmas \ref{lem.b.w} and \ref{V.bw.parl} that
\begin{lemma}\label{max.V}
It holds that
\begin{align*}
\max_{z\in[0,x]}\{z+J_{\Psi}(x-z;\pi^{b_{\Psi}})-J_{\Psi}(x;\pi^{b_{\Psi}})\}=
\begin{cases}
~~~~~~~~~~~~~~~~~~~0, & x\in(0,b_{\Psi}),\\[0.6em]
x-b_{\Psi}+J_{\Psi}(b_{\Psi};\pi^{b_{\Psi}})-J_{\Psi}(x;\pi^{b_{\Psi}}), & x\in[b_{\Psi},\infty).
\end{cases}
\end{align*}
\end{lemma}
{
Actually, a necessary condition for the double barrier strategy with dividend barrier $b\geq 0$ and capital injection barrier $0$ to be the optimal strategy of the control problem \eqref{eq:auxiliarycontrol} (or, equivalently, \eqref{eq:auxiliarycontrol2}), is that (i) when the current controlled surplus $x\leq b$, the optimal pay of paying dividends is to pay no dividends; and, (ii) when  the current controlled surplus $x>b$, the optimal way of paying dividends is to pay all excess surplus exceeding $b$ (i.e., $x-b$) as dividends.
Lemma \ref{max.V} states that, the double barrier strategy with dividend barrier $b_{\Psi}$ and capital injection barrier $0$, satisfies this necessary condition. Furthermore, by combining Lemmas \ref{3.9.add.new.l} and \ref{max.V}, one gets
$$\displaystyle ({\cal A}-q)J_{\Psi}(x;\pi^{b_{\Psi}})+\lambda \Psi(x)+\gamma \max_{z\in[0,x]}\{z+J_{\Psi}(x-z;\pi^{b_{\Psi}})-J_{\Psi}(x;\pi^{b_{\Psi}})\}=0,\quad x\in\R_{+},$$
where, the term ``$\gamma \max_{z\in[0,x]}$'' reminds us that dividends can be paid only at the arrival epochs $(T_{n})_{n\geq 1}$ of the Poisson process $N$ (whose jump intensity is $\gamma$). In the extreme case $\gamma=0$, no dividends will be paid.}

Putting together Lemma \ref{lem.par.b} and Lemmas \ref{V.bw.parl}-\ref{max.V}, we can easily verify, in the following theorem, the conjecture that the double barrier
( Poissonian-continuous-reflection)
dividend and capital injection strategy with dividend barrier $b_{\Psi}$ and capital injection barrier $0$ is the optimal strategy for the auxiliary problem \eqref{eq:auxiliarycontrol}.

\begin{theorem}\label{them.3.1}
Let $(D^{0,b_{\Psi}},R^{0,b_{\Psi}})=(D_t^{0,b_{\Psi}},R_t^{0,b_{\Psi}})_{t\geq0}$ be the strategy defined by \eqref{eq:DRb} with $b$ replaced as $b_{\Psi}$ (see, \eqref{eq:bpsi}). Then, $(D^{0,b_{\Psi}},R^{0,b_{\Psi}})$ is an optimal strategy for the auxiliary problem \eqref{eq:auxiliarycontrol}, i.e.
\begin{align*}
  V_{\Psi}(x)=\sup_{\pi\in\Pi}J_{\Psi}(x;\pi)=J_{\Psi}(x;\pi^{b_{\Psi}}),\quad x>0.
\end{align*}
\end{theorem}

\section{Proof of Theorem~\ref{thm2.1}}\label{sec:proofthm2.1}

The previous section focuses on the solvability of the auxiliary control problem \eqref{eq:auxiliarycontrol}. By applying the results obtained in Section \ref{sec:AuSingCP}, this section is to prove Theorem~\ref{thm2.1} as the main result of this paper.

As preparations, let us consider the following space of functions:
\begin{align}\label{eq:spaceB}
{\cal B}:=\{f:\R_+\times I\to\R;~f(\cdot,i)\in C(\R_+),~\forall i\in I,\text{ and }{\|f\|<\infty}\}
\end{align}
with {$\|f\|:=\max_{i\in I}\sup_{x\in\R_+}\frac{|f(x,i)|}{1+|x|}$.} Then, $({\cal B},\|\cdot\|)$ is a Banach space. Using \eqref{eq:spaceB}, it is easy to verify that
\begin{lemma}\label{lem:finB}
For $f\in{\cal B}$ and $(x,i)\in\R_+\times I$, define
\begin{align}\label{eq:hatf00}
\widehat{f}(x,i):=\sum_{j\in I,j\neq  i}\frac{\lambda_{ij}}{\lambda_{i}}\int_{-\infty}^{0}\left\{f(x+y,j)\mathbf{1}_{\{-y\leq x\}}+(\phi(x+y)+f(0,j))\mathbf{1}_{\{-y>x\}}\right\}F_{ij}(dy),
\end{align}
where, we recall that, $(\lambda_{ij})_{i,j\in I}$ is the generator of the continuous-time Markov chain  $Y$ with finite state space $I$, $\lambda_{i}:=\sum_{j\neq i}\lambda_{ij}$, $F_{ij}(dy)$ is the distribution function of $J_{ij}$,  {and, $J_{ij}$ is the downward jump of $X$ at the instant when $Y$ switches from state $i$ to $j$. } 
Then, $\widehat{f}\in{\cal B}$ for any $f\in{\cal B}$.
\end{lemma}
{Recall that the value function $V(x,i)$ of the primal  Poissonian dividend control problem with capital injection is defined in \eqref{eq:valuefcn}. Let $\|\cdot\|_{\infty}$ be the supremum norm that is defined as
$$\|f-g\|_{\infty}:=\max_{i\in I}\sup_{x\geq0}|f(x,i)-g(x,i)|,\quad f,g\in\mathcal{B}.$$
}
\begin{lemma}\label{v.in.B}
$V\in{\cal B}$. { Moreover, there exists two functions $\underline{V},\overline{V}\in \mathcal{B}$ such that $\|\underline{V}-V\|_\infty<\infty$ and $\|\overline{V}-V\|_\infty<\infty$. }
\end{lemma}

\begin{proof}
Let $\underline{r}:=\min_{i\in I}r_i$, $\overline{r}:=\max_{i\in I}r_i$, $\overline{X}_t:=\sup_{s\leq t}X_s$ and $\underline{X}_t:=\inf_{s\leq t}X_s$ for all $t\geq0$. Recall that $\tau_{i}$ is the first time when the Markov chain $Y$ switches its regime state from $Y_0=i$ to other states. We can derive an upper bound of $V(x,i)$ by considering the extreme case where the manager of the company pays every dollar accumulated by $X$ as dividends as early as possible, i.e., 
$$D_t:=\sum_{n=1}^{\infty}\left(\sup_{s\in[T_{n-1},T_{n}]}\left(X_{s}-X_{T_{n-1}}\right)\vee0\right)\mathbf{1}_{\{T_{n}\leq t\}}+x\mathbf{1}_{\{T_{1}\leq t\}},\quad t\geq 0,$$
and, cover all deficits by capital injection, i.e., $$R_t:=-\inf_{s\leq t}(X_s-D_s)\wedge0,\quad t\geq 0.$$
Obviously, the surplus process
$$U_t=X_t-D_t-\inf_{s\leq t}(X_s-D_s)\wedge0,\quad t\geq 0,$$
takes non-negative values. 
We point out that $D_t$ amounts to the maximum reasonable amount of dividends paid until time $t\geq0${, since any additional dividends paid at time $s\geq0$ must be covered by the capitals injected at time $s$, and, capital injection is costly.} Therefore, we have
\begin{align}
\label{V.upperbar}
 V(x,i)\leq&\, \overline{V}(x,i)
 \nonumber\\
 :=&x\Ex\left[e^{-\int_{0}^{T_{1}}r_{Y_{t}}d t}\right]+\sum_{n=1}^{\infty}\Ex_{x,i}\left[e^{-\int_0^{T_{n}}r_{Y_{t}}dt}
\left(\sup_{s\in[T_{n-1},T_{n}]}\left(X_{s}-X_{T_{n-1}}\right)\vee0\right)\right]
\nonumber\\
=&
x\Ex\left[e^{-\int_{0}^{T_{1}}r_{Y_{t}}d t}\right]+\sum_{n=1}^{\infty}\Ex_{0,i}\left[e^{-\int_0^{T_{n}}r_{Y_{t}}dt}
\left(\sup_{s\in[T_{n-1},T_{n}]}\left(X_{s}-X_{T_{n-1}}\right)\vee0\right)\right].   
\end{align}
Similarly, we can also derive a lower bound by considering another extreme case, where, the manager of the company injects capitals to keep the surplus over $x$ up to the first Poisson observation time, i.e., 
$$R_t:=-\inf_{s\leq t}(X_s-x)\wedge0, \quad t\leq T_{1},$$
and, pays whatever he has as dividends at the first Poisson observation time, i.e., 
$$D_{T_{1}}-D_{T_{1}-}:=X_{T_{1}}-\inf_{s\leq T_{1}}(X_s-x)\wedge0;$$
and, after $T_{1}$, pays no dividends and bails out all deficits by injecting capitals.
Thus,  we have, from the spatial homogeneity, that
\begin{align}
\label{V.lowerbar}
V(x,i)\geq&\, 
\underline{V}(x,i)
\nonumber\\
\hspace{-0.3cm}:=&
\Ex_{x,i}\left[e^{-\int_{0}^{T_{1}}r_{Y_{t}}dt}(X_{T_{1}}-(\underline{X}_{T_{1}}-x)\wedge0)+\phi\int_0^{T_{1}}e^{-\int_{0}^{t}r_{Y_{s}}ds}d((\underline{X}_t-x)\wedge0)\right]\nonumber
\\
&
+\Ex_{x,i}\left[\phi\int_0^{\infty}e^{-\int_{0}^{t+T_{1}}r_{Y_{s}}ds}d\left(\inf_{0\leq s\leq t}(X_{s+T_{1}}-X_{T_{1}})\wedge0\right)\right]\nonumber\\
=&\Ex_{0,i}\left[e^{-\int_{0}^{T_{1}}r_{Y_{t}}dt}(X_{T_{1}}-\underline{X}_{T_{1}}\wedge0)+\phi\int_0^{T_{1}}e^{-\int_{0}^{t}r_{Y_{s}}ds}d(\underline{X}_t\wedge0)\right]
\nonumber\\
&
+\Ex_{0,i}\left[\phi\int_0^{\infty}e^{-\int_{0}^{t+T_{1}}r_{Y_{s}}ds}d\left(\inf_{0\leq s\leq t}(X_{s+T_{1}}-X_{T_{1}})\wedge0\right)\right]
+x\Ex\left[
e^{-\int_{0}^{T_{1}}r_{Y_{t}}dt}
\right].
\end{align}

{ By \eqref{V.upperbar}
and \eqref{V.lowerbar}, one sees that $\overline{V}$ and $\underline{V}$ shares a same term $x\Ex\left[
e^{-\int_{0}^{T_{1}}r_{Y_{t}}dt}
\right]$, which is also the only one term (in the expressions of $\overline{V}$ and $\underline{V}$) that is dependent on $x$.
Consequently, both $\overline{V}$ and $\underline{V}$ are bounded under the norm $\|\cdot\|$, i.e., $V\in\mathcal{B}$. 
Furthermore, one can easily get
$$\max\{\|\underline{V}-V\|_\infty,\|\overline{V}-V\|_\infty\}\leq \|\overline{V}-\underline{V}\|_\infty<\infty.$$
}
This completes the proof.
\end{proof}

For a vector $\overrightarrow{b}=(b_{i})_{i\in I}$ with $b_i\geq0$ for all $i\in I$, recall that $V_{0,\overrightarrow{b}}(x,i)$ (see, \eqref{eq:valuefcn...}) stands for the expected present value of the accumulated differences between dividends and the costs of capital injections of the regime-modulated  Poissonian-continuous-reflection dividend and capital injection strategy with dynamic dividend barrier $b_{Y_{t}}$ at time $t$ and constant capital injection barrier $0$, i.e., the strategy $(D_t^{0,\overrightarrow{b}},R_t^{0,\overrightarrow{b}})_{t\geq 0}$ (see, \eqref{eq:DbRb}). Define the following operator $\mathcal{T}_{\overrightarrow{b}}$ acting on $\mathcal{B}$ as follows: for any $ f\in{\cal B}$, put
\begin{align}\label{def.Tb}
\mathcal{T}_{\overrightarrow{b}}f(x,i)
&:=\Ex_{x}^{i}\left[\sum_{n=1}^{\infty}e^{-q_{i}T_n}\Delta D_{T_n}^{0,b_{i},i}
-\phi\int_{0}^{\infty}e^{-q_{i}t} d R_{t}^{0,b_{i},i}+\lambda_{i}\int_{0}^{\infty}e^{-q_{i}t}\widehat{f}(U_{t}^{0,b_{i},i},i)dt\right],
\end{align}
where, $q_i:=r_{i}+\lambda_i$, $\Ex_x^i$ is the expectation operator conditioned on $X_{0}^{i}=x$, the process
\begin{align*}
U_{t}^{0,b_{i},i}=X_t^{i}-D_{t}^{0,b_{i},i}+R_{t}^{0,b_{i},i},\quad  t\geq0.\nonumber
\end{align*}
is the controlled process with underlying process $X^{i}$ and control strategy $(D_{t}^{0,b_{i},i}, R_{t}^{0,b_{i},i})$ which is defined as
\begin{align*}
D_t^{0,b_{i}, i} &:= \sum_{n=1}^{\infty}((U_{T_{n}-}^{0,b_{i}, i}+\Delta X_{T_{n}}^{i}-b_{i})\vee0)\mathbf{1}_{\{T_{n}\leq t\}},\quad
R_t^{0,b_{i}, i}:=-\inf_{s\leq t}((X_s^{i}-D_s^{0,b_{i}, i})\wedge0),\quad t\geq0.
\end{align*}
In what follows, the scale functions of $X^{i}$ will be denoted by $W_{q,i}$, $Z_{q,i}$ and $\overline{Z}_{q,i}$, whose definitions are given in Appendix~\ref{sec:AppendixA} where the subscript $i$ is absent.

By Proposition \ref{V.x}, Lemma \ref{lem.w} in Appendix~\ref{sec:AppendixB}, the fact that $\widehat{f}\in\mathcal{B}$ for any $f\in{\cal B}$, and the definition \eqref{def.Tb} of the operator ${\cal T}_{\overrightarrow{b}}$, it is not difficult to verify that ${\cal T}_{\overrightarrow{b}}f\in{\cal B}$ for any $f\in{\cal B}$. We then have
\begin{lemma}\label{v.tv}
For any $\overrightarrow{b}\in\R_+^{I}$, the value function $\R_+\times I\ni (x,i)\mapsto V_{0,\overrightarrow{b}}(x,i)$ (see, \eqref{eq:valuefcn...}) is a fixed point of the operator ${\cal T}_{\overrightarrow{b}}$, i.e.,  $V_{0,\overrightarrow{b}}=\mathcal{T}_{\overrightarrow{b}}V_{0,\overrightarrow{b}}$.
\end{lemma}

\begin{proof}
We first prove $V_{0,\overrightarrow{b}}\in{\cal B}$.
Recall that $\tau_{i}$ is the first time when the Markov chain $Y$ switches its regime state from $Y_0=i$ to other states. Hence, $\tau_{i}$ is exponentially distributed with parameter $\lambda_{i}:=\sum_{j\neq i}\lambda_{ij}$, and during the time period $[0,\tau_{i})$, the Markov chain $Y$ stays in the state $i$, and the underlying process $X$ is exactly the single SPLP $X^{i}$. It follows from the strong Markov property and the independence among $(X^i)_{i\in I}$, $Y$ and $(J_{ij})_{i,j\in I}$ that
\begin{align}\label{V.0b.F}
&V_{0,\overrightarrow{b}}(x,i)
=
\Ex_{x,i}\left[\sum_{n=1}^{\infty}e^{-r_{i}T_n}\Delta D^{{0},b_{i},i}_{T_n}\mathbf{1}_{\{T_n\leq \tau_i\}}
-\phi\int_{0}^{\tau_i}e^{-r_{i}t}dR_{t}^{0,b_{i},i}+e^{-r_{i}\tau_i}V_{0,\overrightarrow{b}}(U_{\tau_i}^{0,b_{i},i}+J_{iY_{\tau_i}},Y_{\tau_i})\right]
\nonumber\\
&=\Ex_x^i\left[\sum_{n=1}^{\infty}e^{-q_{i}T_n}\Delta D^{{0},b_{i},i}_{T_n}
-\phi\int_{0}^{\infty}e^{-q_{i}t}dR_{t}^{{0},b_{i},i}\right]+\sum_{j\neq i}\lambda_{ij}\Ex_x^i\left[\int_0^{\infty}e^{-q_{i}t}V_{0,\overrightarrow{b}}(U_{t}^{0,b_{i},i}+J_{ij},j)dt\right]
\nonumber\\
&=
-\frac{\gamma}{q_i+\gamma}\left[\overline{Z}_{q_i,i}(b_i-x)+\frac{\psi_i^{\prime}(0+)}{q_i}\right]+\left[Z_{q_i,i}(b_i-x,\Phi_{q_i+\gamma})+\frac{\gamma}{q_i}Z_{q_i,i}(b_i-x)\right]
\nonumber\\
&\quad
\times\frac{\left(\gamma Z_{q_i,i}(b_i)-\phi(q_i+\gamma)\right)}{(q_i+\gamma)\Phi_{q_i+\gamma}Z_{q_i,i}(b_i,\Phi_{q_i+\gamma})}
+\sum_{j\neq i}\frac{\lambda_{ij}}{q_i}\int_{-\infty}^0\bigg[\int_{0+}^{\infty}V_{0,\overrightarrow{b}}(y+z,j)\Px_{x}^{i}(U_{e_{q_i}}^{0,b_i,i}\in dy)
\nonumber\\
&\quad
+V_{0,\overrightarrow{b}}(z,j)\Px_{x}^{i}(U_{e_{q_i}}^{0,b_i,i}=0)\bigg]F_{ij}(dz),\quad\forall x\geq0,\nonumber\\
&V_{0,\overrightarrow{b}}(x,i)=
\phi x+V_{0,\overrightarrow{b}}(0,i),\quad \forall x<0.
\end{align}
Above, $e_{q_i}$ is an independent exponentially distributed random variable with mean $1/q_i$, and $\Px_x^i$ denotes the expectation operator conditioned on $X_{0}^{i}=x$.
Using the expression of $J_{\Psi\equiv 0}(\cdot;\pi^{b_i})$, the boundedness of $V_{0,\overrightarrow{b}}$ under norm $\|\cdot\|$ in Lemma \ref{v.in.B}, Lemma \ref{lem.w} in Appendix~\ref{sec:AppendixB}, and, the fact that $\max_{j\neq i}\Ex|J_{ij}|<+\infty$, we can deduce that $V_{0,\overrightarrow{b}}\in\mathcal{B}$.

 We next prove that $V_{0,\overrightarrow{b}}\in\mathcal{B}$ is a fixed point of the operator ${\cal T}_{\overrightarrow{b}}$. In fact, using the definition of $\mathcal{T}_{\overrightarrow{b}}$ in \eqref{def.Tb}, the 2nd equality in
\eqref{V.0b.F},
 the independence among $U_t^{0,b_i,i}$, $J_{ij}$ for all $i,j\in I$, and the fact that
\begin{align*}
V_{0,\overrightarrow{b}}(U_t^{0,b_i,i}+J_{ij},j)&=V_{0,\overrightarrow{b}}(U_t^{0,b_i,i}+J_{ij},j)\mathbf{1}_{\{U_t^{0,b_i,i}\geq-J_{ij}\}}
\nonumber\\
&\quad \,\,+(V_{0,\overrightarrow{b}}(0,j)+\phi(U_t^{0,b_i,i}+J_{ij}))\mathbf{1}_{\{U_t^{0,b_i,i}<-J_{ij}\}},
\end{align*}
we can conclude that $V_{0,\overrightarrow{b}}(x,i)=\mathcal{T}_{\overrightarrow{b}}V_{0,\overrightarrow{b}}(x,i)$. The proof is completed.
\end{proof}

In what follows, let us introduce the $n$-times iteration of the operator ${\cal T}_{\overrightarrow{b}}:{\cal B}\to{\cal B}$  as
\begin{align}\label{eq:nitTb}
    {\cal T}_{\overrightarrow{b}}^n(f):={\cal T}_{\overrightarrow{b}}({\cal T}_{\overrightarrow{b}}^{n-1}(f)),\,\, n\geq2,\quad {\cal T}_{\overrightarrow{b}}^1(f):={\cal T}_{\overrightarrow{b}}(f),\quad f\in{\cal B}.
\end{align}

{
\begin{lemma}
\label{lem.4.4.new}
For $f,g\in\mathcal{B}$ satisfying $\|f-g\|_\infty<\infty$, we have
\begin{align}\label{V.b.lim}
\|\mathcal{T}_{\overrightarrow{b}}f-\mathcal{T}_{\overrightarrow{b}}g\|_\infty\leq \beta \|f-g\|_\infty,
\end{align}
where $\beta:=\max_{i\in I}\Ex_0^i[e^{-r_i\tau_{i}}]\in(0,1)$. Moreover, for any $f\in \mathcal{B}$ satisfying $\|f-V_{0,\overrightarrow{b}}\|_\infty<\infty$, it holds that
$$V_{0,\overrightarrow{b}}=\lim_{n\to\infty}\mathcal{T}_{\overrightarrow{b}}^{n}f,\quad \text{in } \|\cdot\|_{\infty}.$$
\end{lemma}

\begin{proof}
For any $f,g\in\mathcal{B}$ satisfying $\|f-g\|_\infty<\infty$, we have
\begin{align}\label{rho.fg}
\left\|\mathcal{T}_{\overrightarrow{b}}f-\mathcal{T}_{\overrightarrow{b}}g\right\|_{\infty}
&=
\max_{i\in I}\sup_{x\in\R_+}\Ex_x^i\Bigg[e^{-r_i\tau_{i}}\sum_{j\in I,j\neq i}\frac{\lambda_{ij}}{\lambda_i}\int_{-\infty}^{-U_{\tau_{i}}^{0,b_i,i}}{|f(0,j)-g(0,j)|}F_{ij}(dy)
\nonumber\\
&\quad+e^{-r_i\tau_{i}}\sum_{j\in I,j\neq i}\frac{\lambda_{ij}}{\lambda_i}\int_{-U_{\tau_{i}}^{0,b_i,i}}^0{|f(U_{\tau_{i}}^{0,b_i,i}+y,j)-g(U_{\tau_{i}}^{0,b_i,i}+y,j)|}F_{ij}(dy)\Bigg]\nonumber\\
&\leq 
\max_{i\in I}\sup_{x\in\R_+}\Ex_x^i\Bigg[e^{-r_i\tau_{i}}\sum_{j\in I,j\neq i}\frac{\lambda_{ij}}{\lambda_i}\int_{-\infty}^{-U_{\tau_{i}}^{0,b_i,i}}\|f-g\|_{\infty}F_{ij}(dy)
\nonumber\\
&\quad+e^{-r_i\tau_{i}}\sum_{j\in I,j\neq i}\frac{\lambda_{ij}}{\lambda_i}\int_{-U_{\tau_{i}}^{0,b_i,i}}^0\|f-g\|_{\infty}F_{ij}(dy)\Bigg]
\nonumber\\
&=\beta\|f-g\|_{\infty},
\end{align}
where
\begin{align}\label{beta.def.}
    \beta:=\max_{i\in I}\Ex_0^i[e^{-r_i\tau_{i}}]\in(0,1). 
\end{align}
Moreover, for any $f\in \mathcal{B}$ satisfying $\|f-V_{0,\overrightarrow{b}}\|_\infty<\infty$,  it holds that
$$\lim_{n\rightarrow\infty}\|\mathcal{T}^{n}_{\overrightarrow{b}}f-V_{0,\overrightarrow{b}}\|_{\infty}=\lim_{n\rightarrow\infty}\|\mathcal{T}^{n}_{\overrightarrow{b}}f-\mathcal{T}^{n}_{\overrightarrow{b}}V_{0,\overrightarrow{b}}\|_{\infty}\leq\lim_{n\rightarrow\infty}\beta^n\|f-V_{0,\overrightarrow{b}}\|_{\infty}=0,$$
which yields $V_{0,\overrightarrow{b}}=\lim_{n\rightarrow\infty}\mathcal{T}^n_{\overrightarrow{b}}f$.
\end{proof}

}
By Lemma \ref{lem:finB}, $\widehat{f}\in{\cal B}$ for any $f\in{\cal B}$. We introduce the following space:
\begin{align}\label{eq:setC}
\mathcal{C}:=\{f\in\mathcal{B};~x\to\widehat{f}(x,i)\text{ is concave,~}\widehat{f^{\prime}}(0,i)\leq\phi\text{ and }\widehat{f^{\prime}}(\infty,i)\in[0,1],~\forall \,\,i\in I\}.
\end{align}
The next lemma shows that the set ${\cal C}$ is nonempty.
\begin{lemma}\label{lem4.4}
Suppose that $f\in\mathcal{B}\cap C^{1}(\mathbb{R}_{+})$ is concave and non-decreasing, and, satisfies $f^{\prime}(0,i)\leq \phi$ and $f^{\prime}(\infty,i)\in[0,1]$ for all $i\in I$. Then $f\in \mathcal{C}$.
\end{lemma}

\begin{proof}
By definition of $\widehat{f}$, we may rewrite
\begin{align*}
\widehat{f}(x,i)&=
\sum_{j\in I, j\neq i}\frac{\lambda_{ij}}{\lambda_{i}}\left[\int_{-x}^{0}(f(x+y,j)-(\phi(x+y)+f(0,j)))F_{ij}(dy)+\phi(x+\Ex[J_{ij}])+f(0,j)\right],
\end{align*}
which implies that
\begin{align}\label{f.hat.par}
\widehat{f}^{\prime}(x,i)=\sum_{j\in I, j\neq  i}\frac{\lambda_{ij}}{\lambda_{i}}\left[\phi+\int_{-x}^{0}(f^{\prime}(x+y,j)-\phi)F_{ij}(dy)\right].
\end{align}
Combining the concavity of $x\to f(x,i)$ and \eqref{f.hat.par}, one can get $x\to\widehat{f}(x,i)$ is also concave. Moreover, by the fact that $f^{\prime}(x+y,j)-\phi\leq0$, we can deduce
$\widehat{f}^{\prime}(x,i)\leq\sum_{j\in I, j\neq  i}\frac{\lambda_{ij}}{\lambda_{i}}\phi=\phi$.
On the other hand, using the fact that $f^{\prime}(\infty,i)\in[0,1]$, we have
\begin{align*}
0=\sum_{j\in I, j\neq i}\frac{\lambda_{ij}}{\lambda_{i}}\left[\phi-\int_{-\infty}^{0}\phi F_{ij}(dy)\right]\leq\widehat{f}^{\prime}(\infty,i)\leq\sum_{j\in I, j\neq  i}\frac{\lambda_{ij}}{\lambda_{i}}\left(\phi+\int_{-\infty}^{0}(1-\phi)F_{ij}(dy)\right)=1.
\end{align*}
Hence $f\in\mathcal{C}$. Thus, the proof of the lemma is complete.
\end{proof}

For any $i\in I$, let 
$$U^{i}_{t}:=X^{i}_{t}-D^{i}_{t}+R^{i}_{t},\quad t\geq0,$$
be the controlled surplus process with control $\pi^i:=(D^{i},R^{i})$ and driving process $X^{i}=(X^i_t)_{t\geq0}$. Then, for any $f\in \mathcal{B}$, we define the following operator as: for all $(x,i)\in\R_+\times I$,
\begin{align}\label{T.sup}
\mathcal{T}_{\sup} f(x,i)&:=
\sup_{(D,R)\in\mathcal{U}}\Ex_{x,i}\left[\sum_{n=1}^{\infty}e^{-r_{i}T_n}\Delta D_{T_n}\mathbf{1}_{\{T_n\leq \tau_{i}\}}
-\phi\int_{0}^{\tau_{i}}e^{-r_{i}t}dR_t
+e^{-r_{i}\tau_{i}}\widehat{f}(U_{\tau_{i}-},i)\right]
\nonumber\\
&=
\sup_{\pi^{i}=(D^i,R^i)}\Ex_x^i\left[\sum_{n=1}^{\infty}e^{-q_{i}T_n}\Delta D^{i}_{T_n}-\phi\int_{0}^{\infty}e^{-q_{i}t}dR^{i}_t+\lambda_{i}\int_{0}^{\infty}e^{-q_{i}t}\widehat{f}(U^{i}_t,i)dt\right].
\end{align}
For the functions $\overline{V}(x,i)$ and $\underline{V}(x,i)$ for $(x,i)\in\R_+\times I$ defined by \eqref{V.upperbar} and \eqref{V.lowerbar}, we define iteratively that
\begin{align*}
   \underline{V}_0:=\underline{V},~ &~ \overline{V}_0:=\overline{V},\\
   \underline{V}_n:=\mathcal{T}_{\sup}(\underline{V}_{n-1}),~&~ \overline{V}_n:=\mathcal{T}_{\sup}(\overline{V}_{n-1}),~\forall  n\geq1.
\end{align*}
Then, we have
\begin{lemma}\label{lem.V.n}
It holds that $\underline{V}_n\leq V\leq \overline{V}_n$ on $\mathbb{R}_+\times I$ for all $n\geq1$, and
\begin{align}\label{V.lim}
V(x,i)=\lim_{n\uparrow\infty}\underline{V}_n(x,i)=\lim_{n\uparrow\infty}\overline{V}_n(x,i),\quad (x,i)\in\R_+\times I,
\end{align}
where $\lim_{n\uparrow\infty}\underline{V}_n$ and $\lim_{n\uparrow\infty}\overline{V}_n$ are limits under the {\color{blue}$\|\cdot\|_{\infty}$-norm}. Moreover, we have $V\in\mathcal{C}$.
\end{lemma}

\begin{proof}
One can verify the first claim of Lemma \ref{lem.V.n} by applying the method of induction. In fact, it follows from \eqref{V.upperbar}, \eqref{V.lowerbar}, Lemma \ref{v.in.B} and Lemma \ref{lem4.4} that $\underline{V}_0\leq V\leq\overline{V}_0$ and $\underline{V}_0,\overline{V}_0\in\mathcal{C}$. Assume that $\underline{V}_{n-1}\leq V\leq\overline{V}_{n-1}$, then by definition we have $\widehat{\underline{V}_{n-1}}\leq \widehat{V}\leq\widehat{\overline{V}_{n-1}}$, implying that
\begin{align*}
    \underline{V}_n=\mathcal{T}_{\sup}(\underline{V}_{n-1})\leq\mathcal{T}_{\sup}({V})\leq\mathcal{T}_{\sup}(\overline{V}_{n-1})=\overline{V}_n,
\end{align*}
which, together with the fact that $\mathcal{T}_{\sup}V=V$ (i.e., \eqref{eq:dpp} holds true by the dynamic program), implies that $\underline{V}_{n}\leq V\leq\overline{V}_{n}$ for all $n\geq1$.

For any $f\in\mathcal{C}$ and $i\in I$, Theorem \ref{them.3.1} guarantees the existence of $b^{f}_i\in(0,\infty)$ such that the 2nd equality in \eqref{T.sup} is achieved by
the expected total discounted dividends subtracted by the total discounted
costs of capital injection under a double barrier  Poissonian-continuous-reflection strategy with upper dividend barrier $b^f_i$ and lower capital injection barrier $0$.
Let $\overrightarrow{b}^f=(b^f_i)_{i\in I}$. Then $\mathcal{T}_{\sup}f=\mathcal{T}_{\overrightarrow{b}^f}f$ on $\mathbb{R}_+\times I$. Moreover, by Lemma \ref{lem.par.b} and Lemma \ref{V.bw.parl}, one sees that $\mathcal{T}_{\sup}f(\cdot,i)\in C^1(0,\infty)$ is concave, $(\mathcal{T}_{\sup}f)^{\prime}(0,i)\leq\phi$, and, $(\mathcal{T}_{\sup}f)^{\prime}(\infty,i)\in[0,1]$ for all $i\in I$. This together with Lemma \ref{lem4.4} yields $\mathcal{T}_{\sup}f\in\mathcal{C}$. In addition, for any $f,g\in\mathcal{C}$ such that $\|f-g\|_{\infty}<\infty$, we have
{
\begin{align*}
\left\|\mathcal{T}_{\sup}f-\mathcal{T}_{\sup}g\right\|_{\infty}&=\left\|\mathcal{T}_{\overrightarrow{b}^f}f-\mathcal{T}_{\overrightarrow{b}^g}g\right\|_{\infty}=\left\|\sup_{\overrightarrow{b}}\mathcal{T}_{\overrightarrow{b}}f-\sup_{\overrightarrow{b}}\mathcal{T}_{\overrightarrow{b}}g\right\|_{\infty}\leq\sup_{\overrightarrow{b}}\left\|\mathcal{T}_{\overrightarrow{b}}f-\mathcal{T}_{\overrightarrow{b}}g\right\|_{\infty}\\
&\leq
\beta\|f-g\|_{\infty},
\end{align*}
where $\beta$ is given by \eqref{beta.def.}. From Lemma \ref{v.in.B}, one knows $\overline{V}$ and $\underline{V}$ defined by \eqref{V.upperbar} and \eqref{V.lowerbar} are elements of $\mathcal{C}$ and satisfy $\|\overline{V}-\underline{V}\|_{\infty}<\infty$. Hence, it holds that
$$
\lim_{n\rightarrow\infty}\max\{\|\underline{V}_n-V\|_{\infty},\|\overline{V}_n-V\|_{\infty}\}\leq
\lim_{n\rightarrow\infty}\|\overline{V}_n-\underline{V}_{n}\|_{\infty}\leq \lim_{n\rightarrow\infty}\beta^n\|\overline{V}-\underline{V}\|_{\infty}=0,$$
which implies that \eqref{V.lim} holds true.}


Combining the facts that $\underline{V}_{0},\,\overline{V}_{0}\in\mathcal{C}$, and that $\mathcal{T}_{\sup}$ is a contraction mapping from $\mathcal{C}$ to itself, we know that $(\underline{V}_n)_{n\geq1}\subseteq\mathcal{C}$ and $(\overline{V}_n)_{n\geq1}\subseteq\mathcal{C}$.
To prove the claim that $V\in\mathcal{C}$, we use \eqref{V.lim} and the dominated convergence theorem, to derive that, for all $(x,i)\in\R_+\times I$,
\begin{align}\label{V.hat}
\widehat{V}(x,i)&=
\sum_{j\in I, j\neq i}\frac{\lambda_{ij}}{\lambda_{i}}\int_{-\infty}^{0}\left[V(x+y,j)\mathbf{1}_{\{-y\leq x\}}
+(\phi(x+y)+V(0,j))\mathbf{1}_{\{-y>x\}}\right]F_{ij}(dy)\nonumber\\
&=
\lim_{n\to\infty}\sum_{j\in I, j\neq  i}\frac{\lambda_{ij}}{\lambda_{i}}\int_{-\infty}^{0}\left[\underline{V}_n(x+y,j)\mathbf{1}_{\{-y\leq x\}}
+(\phi(x+y)+\underline{V}_n(0,j))\mathbf{1}_{\{-y>x\}}\right]F_{ij}(dy)\nonumber\\
&=
\lim_{n\to\infty}\sum_{j\in I, j\neq  i}\frac{\lambda_{ij}}{\lambda_{i}}\int_{-\infty}^{0}\left[\overline{V}_n(x+y,j)\mathbf{1}_{\{-y\leq x\}}
+(\phi(x+y)+\overline{V}_n(0,j))\mathbf{1}_{\{-y>x\}}\right]F_{ij}(dy)\nonumber\\
&=
\lim_{n\to\infty}\widehat{\underline{V}}_{n}(x,i)=\lim_{n\to\infty}\widehat{\overline{V}}_{n}(x,i).
\end{align}
By using \eqref{V.hat}, Theorem E in Section 13 of Chapter 1 in \cite{Roberts73}, and the fact that $(\underline{V}_n)_{n\geq1}\subseteq\mathcal{C}$ and $(\overline{V}_n)_{n\geq1}\subseteq\mathcal{C}$, we derive that $V\in\mathcal{C}$.
\end{proof}

With the previous preparations, we can finally give the proof of Theorem \ref{thm2.1}.

\begin{proof}[Proof of Theorem 2.1]
We have from Lemma \ref{lem.V.n} that $V\in\mathcal{C}$. This together with Theorem \ref{them.3.1} yields that, there exists a function $\overrightarrow{b}^V=(b_i^V)_{i\in I}\in(0,\infty)^I$ such that $V(x,i)=\mathcal{T}_{\sup}V(x,i)=\mathcal{T}_{\overrightarrow{b}^V}V(x,i)$ for all $(x,i)\in\mathbb{R}_+\times I$. Thus, we have
\begin{align}
\lim\limits_{n\rightarrow\infty}\|\mathcal{T}_{\overrightarrow{b}^V}^{n}V-V\|_{\infty}=0.\label{4.15.new}
\end{align}
In addition, 
{similar to Lemma \ref{v.in.B}, one can find an upper and a lower bound of form 
$$x\Ex\left[e^{-\int_{0}^{T_{1}}r_{Y_{t}}d t}\right]+\text{a constant},$$
for $V_{0,\overrightarrow{b}}$.
This combined with (the proof of) Lemma \ref{v.in.B} gives
$\|V-V_{0,\overrightarrow{b}}\|_{\infty}<\infty$. Then, we can apply} Lemma \ref{lem.4.4.new} and the fact that $V, V_{0,\overrightarrow{b}}\in\mathcal{B}$ (see, Lemmas \ref{v.in.B} and \ref{v.tv})
to get
\begin{align}
\lim_{n\rightarrow\infty}\|\mathcal{T}_{\overrightarrow{b}^V}^nV-V_{0,\overrightarrow{b}^V}\|_{\infty}=0.\label{4.16.new}
\end{align}
Putting together \eqref{4.15.new} and \eqref{4.16.new} implies that $V(x,i)=V_{0,\overrightarrow{b}^V}(x,i)$ for all $(x,i)\in\mathbb{R}_+\times I$.

{
It remains to prove that, the strategy $(D^{0,\overrightarrow{b}^V}_t,R^{0,\overrightarrow{b}^V}_t)_{t\geq0}$ is an admissible dividend and capital injection strategy. To this end, put $\underline{r}=\min_{i\in I}r_i$, $\underline{b}:=\min_{i\in I}b^V_i$, and let $(\hat{\tau}_n)_{n\geq 0}$  (with $\hat{\tau}_{0}:=0$) be the $n$-th time when the state of $Y$ switches.
For any $(x,i)\in\mathbb{R}_+\times I$, one has
\begin{eqnarray}
\label{4.21}
\hspace{-0.3cm}&&\hspace{-0.3cm}
\mathbb{E}_{x,i}\left[\int_0^{\infty}e^{-\int_0^tr_{Y_s}ds}dR^{0,\overrightarrow{b}^V}_t\right]
\leq
\mathbb{E}_{x,i}\left[\sum_{n=1}^{\infty}\int_{\hat{\tau}_{n-1}}^{\hat{\tau}_n}e^{-\underline{r}t}dR^{0,\overrightarrow{b}^V}_t\right]
\nonumber\\
\hspace{-0.3cm}&\leq&\hspace{-0.3cm}
\mathbb{E}_{x,i}\left[\sum_{n=1}^{\infty}\left(e^{-\underline{r}\hat{\tau}_n}J_{Y_{\hat{\tau}_{n-1}}Y_{\hat{\tau}_n}}+\int_{(\hat{\tau}_{n-1},\hat{\tau}_n)}e^{-\underline{r}t}dR^{0,b^V_{Y_{\hat{\tau}_{n-1}}}}_t\right)\right]
\nonumber\\
\hspace{-0.3cm}&\leq&\hspace{-0.3cm}
\mathbb{E}_{x,i}\left[\sum_{n=1}^{\infty}e^{-\underline{r}\hat{\tau}_n}\right]\max_{i,j\in I, i\neq j}\mathbb{E}[|J_{ij}|]+\mathbb{E}_{x,i}\left[\sum_{n=1}^{\infty}\mathbb{E}_{x,i}\left[\int_{(\hat{\tau}_{n-1},\hat{\tau}_n)}e^{-\underline{r}t}dR^{0,b^V_{Y_{\hat{\tau}_{n-1}}}}_t\bigg|\mathcal{F}_{\hat{\tau}_{n-1}}\right]\right],
\end{eqnarray}
where  
\begin{align}
\label{4.22}
   & \mathbb{E}_{x,i}\left[\sum_{n=1}^{\infty}e^{-\underline{r}\hat{\tau}_n}\right]=\sum_{n=1}^{\infty}\mathbb{E}_{x,i}\left[\prod_{k=1}^{n-1}e^{-\underline{r}\left(\hat{\tau}_k-\hat{\tau}_{k-1}\right)} 
\mathbb{E}_{x,i}\left[\left.e^{-\underline{r}\left(\hat{\tau}_n-\hat{\tau}_{n-1}\right)}\right|{\mathcal{F}_{\hat{\tau}_{n-1}}}\right]\right]
\nonumber\\
=&
\sum_{n=1}^{\infty}\mathbb{E}_{x,i}\left[\prod_{k=1}^{n-1}e^{-\underline{r}\left(\hat{\tau}_k-\hat{\tau}_{k-1}\right)} 
\mathbb{E}_{U^{0,\overrightarrow{b}^V}_{\hat{\tau}_{n-1}},Y_{\hat{\tau}_{n-1}}}\left[e^{-\underline{r}\hat{\tau}_1}\right]\right]
\nonumber\\
\leq&
\sum_{n=1}^{\infty}\mathbb{E}_{x,i}\left[\prod_{k=1}^{n-1}e^{-\underline{r}\left(\hat{\tau}_k-\hat{\tau}_{k-1}\right)} 
\right]\left(\max_{i\in I}\frac{\lambda_{i}}{\lambda_{i}+\underline{r}}\right)
\nonumber\\
\leq&\cdots\leq 
\sum_{n=1}^{\infty}\left(\max_{i\in I}\frac{\lambda_{i}}{\lambda_{i}+\underline{r}}\right)^{n}
=
\left(\max_{i\in I}\frac{\lambda_{i}}{\lambda_{i}+\underline{r}}\right)\left(1-\max_{i\in I}\frac{\lambda_{i}}{\lambda_{i}+\underline{r}}\right)^{-1}<\infty,
\end{align}
and
\begin{align}
\label{4.23}
    &e^{\underline{r}\tau_{n-1}}\mathbb{E}_{x,i}\left[\int_{(\hat{\tau}_{n-1},\hat{\tau}_n)}e^{-\underline{r}t}dR^{0,b^V_{Y_{\hat{\tau}_{n-1}}}}_t\bigg|\mathcal{F}_{\hat{\tau}_{n-1}}\right]
    \nonumber\\
    =&
\mathbb{E}_{x,i}\left[\int_{0}^{\hat{\tau}_n-\hat{\tau}_{n-1}}e^{-\underline{r}t}dR^{0,b^V_{Y_{\hat{\tau}_{n-1}}}}_{t+\hat{\tau}_{n-1}}\bigg|\mathcal{F}_{\hat{\tau}_{n-1}}\right]
    \nonumber\\
    =&
    \mathbb{E}_{U^{0,\overrightarrow{b}^V}_{\hat{\tau}_{n-1}},Y_{\hat{\tau}_{n-1}}}
    \left[\int_{0}^{\hat{\tau}_{1}}e^{-\underline{r}t}dR^{0,b^V_{Y_{\hat{\tau}_{n-1}}}}_{t}\right]
     \nonumber\\
    \leq &
    \max_{i\in I}\mathbb{E}_{0}^{i}
    \left[\int_{0}^{e_{\lambda_{i}}}e^{-\underline{r}t}dR^{0,\underline{b},i}_{t}\right]
    \nonumber\\
    =&
    \max_{i\in I}\mathbb{E}_{0}^{i}
    \left[\int_{0}^{\infty}e^{-(\lambda_{i}+\underline{r})t}dR^{0,\underline{b},i}_{t}\right]
    \nonumber\\
    =&
    \max_{i\in I}\frac{Z_{\lambda_i+\underline{r},i}(\underline{b},\Phi_{\lambda_i+\underline{r}+\gamma,i})+\frac{\gamma}{\lambda_i+\underline{r}}Z_{\lambda_i+\underline{r},i}(\underline{b})}{\Phi_{\lambda_i+\underline{r}+\gamma,i}Z_{\lambda_i+\underline{r},i}(\underline{b},\Phi_{\lambda_i+\underline{r}+\gamma,i})}<\infty,
\end{align}
where Corollary 3.3 of \cite{Avram18} is used in the last line of \eqref{4.23}.
Combining \eqref{4.21}, \eqref{4.22} and \eqref{4.23} yields
$$\mathbb{E}_{x,i}\left[\int_0^{\infty}e^{-\int_0^tr_{Y_s}ds}dR^{0,\overrightarrow{b}^V}_t\right]<
\infty,$$
i.e., the strategy $(D^{0,\overrightarrow{b}^V}_t,R^{0,\overrightarrow{b}^V}_t)_{t\geq0}$ is admissible. Therefore, $\overrightarrow{b}^*:=\overrightarrow{b}^V=(b^V_i)_{i\in I}$ is the desired vector of  Poissonian dividend barrier such that the conclusion of Theorem \ref{thm2.1} holds true.
The proof is complete.}
\end{proof}

{\section{Numerical Illustration}

This section aims to illustrate the features of the optimal strategies. We are
interested in the values of $V(x,1)$ and $V(x,2)$ defined in \eqref{eq:valuefcn}. Since the optimal
strategy for the control problem \eqref{eq:valuefcn} has turned out to be some double barrier dividend and capital injection strategy, the values of the dividend barriers $\overrightarrow{b^*}=(b^*_1, b^*_2)$ corresponding to the optimal strategy are also
quantities of interest. In particular, the impacts of model parameters on the optimal strategy are studied.

For computational convenience, we consider a particular Markov additive process $(X,Y)$, where, the process $Y=(Y_t)_{t\geq0}$ is a continuous time Markov chain with a two-state space $I=\{1,2\}$ and generator matrix $(\lambda_{ij})_{i,j\in I}=(-0.3,0.3;0.6,-0.6)$; and, conditionally on that $Y$ is in state $i\in I$, the process $X$ evolves as 
$$X^i_{t}=\mu_i t
+\sum_{j=1}^{N^i_{t}}e^i_j,\quad t\geq0,$$
where $\mu_i$ is the drift coefficient,  
$(e^i_j)_{j\geq1}$ are independent and identically distributed exponential random variables with mean $1/\eta_i$, and $(N^i_{t})_{t\geq0}$ is a independent Poisson processes with intensity $p_i$. The Laplace exponent of $X^i$ is given by
\begin{eqnarray}
\psi_i(\theta):=\log\mathbb{E}(e^{\theta X^i_1})=
\mu_i\theta+p_i\big[\frac{\eta_i}{\eta_i-\theta}-1\big],\quad \theta\in (-\infty,\eta_i),\,\,i=1,2.
\nonumber
\end{eqnarray}
Assume that, at the instant $Y$ switches its states, the downward jump size $J$ (in $X$) is an independent exponentially distributed random variable with mean $-2$. 

Let us set the market parameters as below (unless otherwise specified):
\begin{table}[H]
\centering
\caption{The values of the parameters}
 \begin{tabular}{lccccccccccc}
\hline\hline
\noalign{\smallskip}
Parameter &$x$ & $\gamma$ & $\mu_1$ & $\mu_2$ & $p_1$ &$p_2$& $\eta_1$ &$\eta_2$ & $r_1$ & $r_2$ &$\phi$ \\\noalign{\smallskip}
Value  &0 & 0.3& -1.5 & -1.7 & 0.2 & 0.3& 0.2 &
0.1 & 0.05 & 0.03 & 1.5 \\\noalign{\smallskip}
\hline
\end{tabular}   
\end{table}
For $i=1,2$, denote $q_i=r_i+\lambda_i$ with $\lambda_i=\sum_{j\neq i}\lambda_{ij}$.
Further define $W_{q_i,i}(x)=0$ for $x<0$.
The scale function $W_{q_i,i}(x)$ reads as
\begin{eqnarray}\label{W_q.example.2}
W_{q_i,i}(x)=\frac{e^{\vartheta_{q_i,i} x}}{\psi_i^{\prime}(\vartheta_{q_i,i})}
+\frac{e^{\Phi_{q_i,i} x}}{\psi_i^{\prime}(\Phi_{q_i,i})},\quad x\in [0,\infty),
\end{eqnarray}
where $\vartheta_{q_i,i}$
and $\Phi_{q_i,i}$ denote the two roots of $\psi_i(\theta)={q_i}$ such that $\Phi_{q_i,i}>\vartheta_{q_i,i}$ and 
$\Phi_{q_i,i}>0$.
Then the scale function $Z_{q_i,i}(x)$ can be written as
\begin{eqnarray}\label{Z_q.example.2}
Z_{q_i,i}(x)=1+{q_i}\int_0^xW_{q_i,i}(y){d}y=\frac{{q_i}e^{\vartheta_{q_i,i} x}}{\psi_{i}^{\prime}(\vartheta_{q_i,i})\vartheta_{q_i,i}}
+\frac{{q_i}e^{\Phi_{q_i,i} x}}{\psi_i^{\prime}(\Phi_{q_i,i})\Phi_{q_i,i}},\quad x\in [0,\infty).
\end{eqnarray}

As shown in Section 4, if a function ${V}_0\in\mathcal{B}$ can be found such that $\|{V}_0-V\|_{\infty}<\infty$, then iterating through the operator $\mathcal{T}_{\sup}$ will yield an approximate value function $V$.
The specific iterative steps are as follows:
\begin{itemize}
    \item[(i)] Put ${V}_0(x,i)=x\mathbb{E}_{0,i}\left[e^{-\int_0^{T_1}r_{i}dt}\right]$, for $i=1,2$. From Lemma \ref{v.in.B}, one knows $V_0\in\mathcal{B}$ and $\|{V}_0-V\|_{\infty}<\infty$.
    \item[(ii)] Assume that one has obtained $V_{n}$ satisfying $\|V_n-V\|_{\infty}<\infty$ and then calculus $\widehat{{V}}_n(x,i)$, $i=1,2,$ by
$$\widehat{{V}}_n(x,i)=\sum_{j\in I,j\neq  i}\frac{\lambda_{ij}}{\lambda_{i}}\int_{-\infty}^{0}\left\{{V}_n(x+y,j)\mathbf{1}_{\{-y\leq x\}}+(\phi(x+y)+{V}_n(0,j))\mathbf{1}_{\{-y>x\}}\right\}F_{ij}(dy).$$
    \item[(iii)] Compute $J_{\widehat{{V}}_n}(x,i;\pi^{b_i})$, $i=1,2,$ by using expression of \eqref{eq:JPhipib-exp} in Proposition \ref{V.x} with $\overrightarrow{b}=(b_1,b_2)\in\mathbb{R}^2_+$ and
    $$ J_{\widehat{{V}}_n}(x,i;\pi^{b_i}):=\mathcal{T}_{\overrightarrow{b}}{V}_n(x,i)=\Ex_x\left[\int_0^{\infty}e^{-q_it}d(D^{0,b_i,i}_t-\phi R_t^{0,b_i,i})+\lambda\int_0^{\infty}e^{-q_it}\widehat{{V}}_n(U^{0,b_i,i}_{t},i)dt\right].$$
    \item[(iv)] Determine $\overrightarrow{b}^{\widehat{{V}}_n}=(b_1^{\widehat{V}_n},b_2^{\widehat{V}_n})$ by
    $$b_i^{\widehat{{V}}_n}=\inf\left\{b\geq0;~\frac{d}{dx}J_{\widehat{V}_n}(x,i;\pi^{b})|_{x=b}\leq 1\right\},\quad i=1,2.$$    
    \item[(v)] Judge whether $\|V_n-V_{n-1}\|_{\infty}$ is less than $\epsilon=0.01$. If so, put an end to the step and one derives two approximations $V(x,i)\approx J_{\widehat{V}_n}(x,i;\pi^{{b}_i^{\widehat{V}_n}})$ as well as $\overrightarrow{b^*}\approx (b_1^{\widehat{V}_n},b_2^{\widehat{V}_n})$. If not, put $V_{n+1}(x,i)= J_{\widehat{V}_n}(x,i,\pi^{{b}_i^{\widehat{V}_n}})$, $i=1,2$. From Lemma \ref{lem.4.4.new}, one knows $\|V_{n+1}-V\|_{\infty}\leq \beta\|V_n-V\|_{\infty}<\infty.$
    \item[(vi)] Go back to step (ii) with $V_n$ replaced by $V_{n+1}$.
\end{itemize}

We remark that all expectations in iterative steps are computed using Monte-Carlo method and all numerical calculations are performed using MATLAB.

The functions $J_{\widehat{V}_n}(x,1;\pi^{b^{\widehat{V}_n}_1})$ and $J_{\widehat{V}_n}(x,2;\pi^{b^{\widehat{V}_n}_2})$ (with $n\geq 1$) obtained in each iteration step are demonstrated in the left and right panels of Figure \ref{figure1}, respectively. From numerical results, we know that iteration ends at 64-th time and the dividend barrier convergences to $\overrightarrow{b^*}=(3.81,3.97)$.

\begin{figure}[H]
\centering
\includegraphics[width=3.1in]{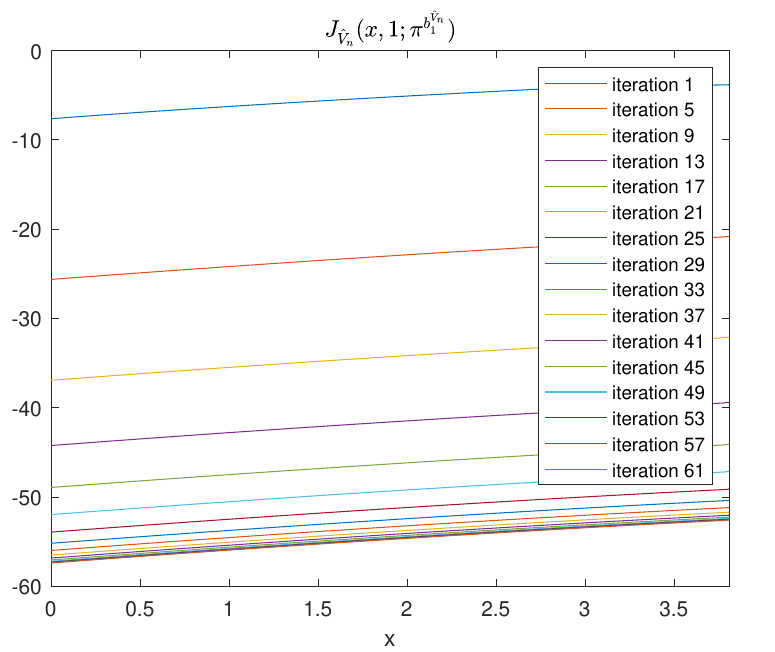}
\quad
\includegraphics[width=3.1in]{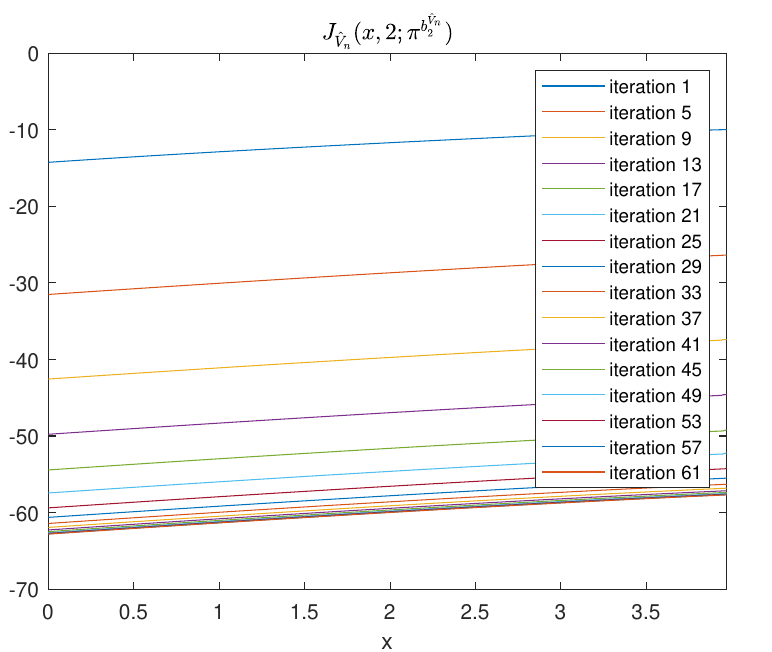}
\caption{Iteration of functions $J_{\widehat{V}_n}(x,1;\pi^{b^{\widehat{V}_n}_1})$ and $J_{\widehat{V}_n}(x,2;\pi^{b^{\widehat{V}_n}_2})$ for $1\leq n \leq 61$.}
\label{figure1}
\end{figure}

The simulated trajectories of the uncontrolled stochastic process $X_t$ are provided in the left panel of Figure \ref{figure2}. The blue line represents segments of the path of the process $X$ (which evolves as the compound Poisson process with drift $X^1$) when $Y$ is in the state $1$; and, the orange line represents segments of the path of the process $X$ (which evolves as the compound Poisson process with drift $X^2$) when $Y$ is in the state $2$. The vertical dotted line is a downward jump when state $Y$ switches to another state. We observe that the process $X$ may take negative values if there is no capital injection, and, may take values larger than $\overrightarrow{b^*}=(3.81,3.97)$ if there is no dividend payments.

The simulated trajectories of the controlled surplus process $U$ under the optimal dividend and capital injection strategy, are shown in the right panel of Figure \ref{figure2}. The horizontal blue dashed line indicates the optimal dividend barrier $b^*_1$ when $Y$ is in the state $1$; and the horizontal orange dashed line represents the optimal dividend barrier $b^*_2$ when $Y$ is in the state $2$. The solid blue (resp., orange) line represent segments of the path of the process $U$ when $Y$ is in the state $1$ (resp., $2$).
It can be seen that when the controlled surplus process $U$ is about to exceed the barrier $\overrightarrow{b^*}$, the agent will withdraw the overshoots (in form of dividends) so that the surplus process remains at $\overrightarrow{b^*}$; when the surplus process is about to go below $0$, the agent will inject the minimal amount of capitals to keep the surplus process non-negative.

\begin{figure}[H]
\centering
\includegraphics[width=3.1in]{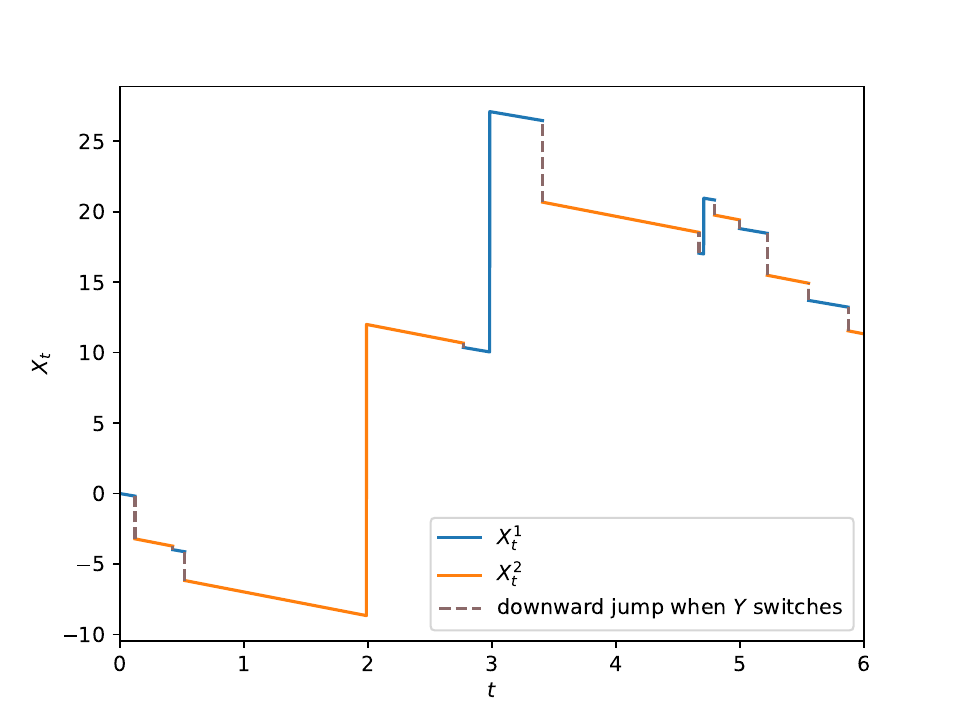}
\quad
\includegraphics[width=3.1in]{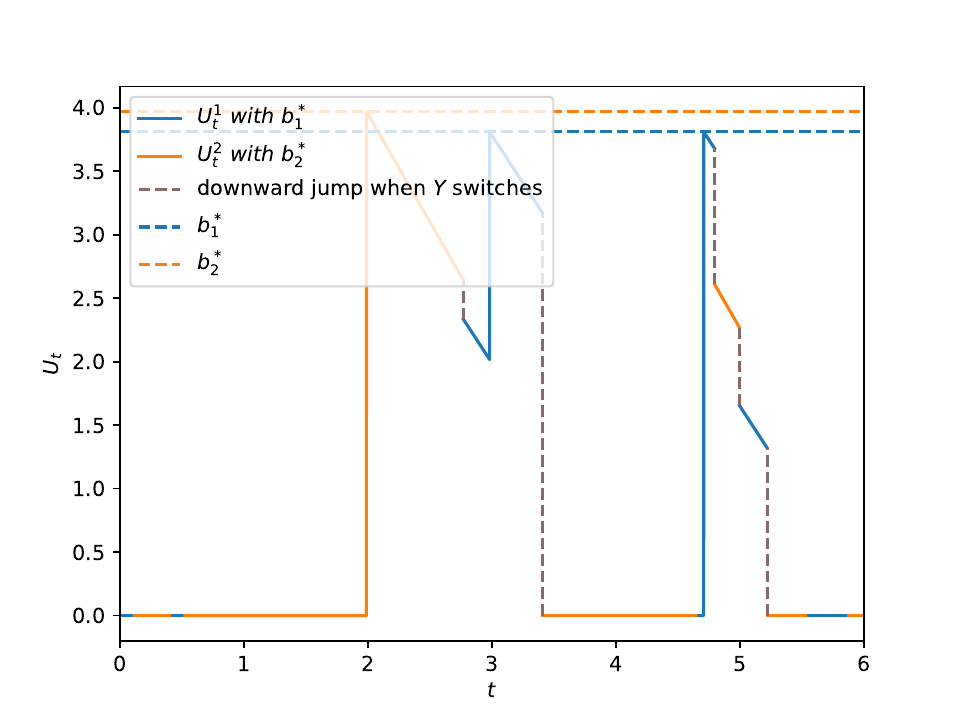}
\caption{The uncontrolled surplus process $X_t$ and the controlled surplus process $U_t$ with the dividend barriers $\ora{b^*}=(3.81,3.97)$.}
\label{figure2}
\end{figure}

In the following, we conduct a sensitivity analysis with respect to parameters $\gamma$, $\mu_1$, $p_1$, $\eta_1$ and $\phi$ to examine the influence of different parameter settings on the optimal dividend barrier $\overrightarrow{b^*}$.

Recall that $\gamma$ is the intensity of the Poisson process $N=(N_t)_{t\geq0}$, and, dividend payments are made only at the arrival epochs of $N$. Table \ref{tab2} shows that the optimal dividend barrier $\overrightarrow{b^*}=(b^*_1,b^*_2)$ tends to increase as $\gamma$ increases from $0.3$ to $0.7$. This is reasonable, because: (i) the larger $\gamma$ becomes, the more frequently dividends are paid, and hence, more dividends may be paid; (ii) the more dividends are deducted away from the surplus process, the more likely the surplus process will hit $0$, leading to more injected capitals; (iii) injecting capitals is costly/expensive, excessive injected capitals is unfavorable to attaining the supremum in \eqref{eq:valuefcn}; (iv) 
one has to make a trade-off between increasing the dividend barrier (to have less capitals injected) and decreasing the dividend barrier (to get more dividends); (v) 
Table \ref{tab2} reveals that increasing $\gamma$ from $0.3$ to $0.7$ really results in excessive injected capitals, and the dividend barrier has to be increased so that the supremum in \eqref{eq:valuefcn} is attainable.

\begin{table}[H]
\centering
\caption{$\ora{b^*}$ with respect to the intensity $\gamma$ of Poisson process}
\label{tab2}
 \begin{tabular}{lccccc}
\hline\hline
\noalign{\smallskip}
$\gamma$ & 0.3 & 0.4 & 0.5 & 0.6 & 0.7  \\\noalign{\smallskip}
$b^*_1$ & 3.81 & 3.98 & 4.22 & 4.51 & 4.80 \\\noalign{\smallskip}
$b^*_2$  & 3.97 & 4.08 & 4.29 & 4.56 & 4.84
\\\noalign{\smallskip}
\hline
\end{tabular}   
\end{table}

Table \ref{tab3} demonstrates that the optimal dividend barrier $\overrightarrow{b^*}=(b^*_1,b^*_2)$ tends to increase as $\phi$ increases from $1.05$ to $2$. This is a natural phenomenon, because: (i) the larger $\phi$ is, the more costly/expensive capital injection is,  hence, it would be more favorable (in the sense of achieving the supremum of \eqref{eq:valuefcn}) to inject less capitals for larger $\phi$; (ii) keeping the dividend barriers unchanged or decreasing the dividend barriers will probably result in excessive injected capitals;
(iii) Table \ref{tab3} reveals that increasing $\phi$ from $1.05$ to $2$ really results in excessive injected capitals, thereby, $\overrightarrow{b^*}=(b^*_1,b^*_2)$ has to be increased to eliminate excessive injected capitals.  

\begin{table}[H]
\centering
\caption{$\ora{b^*}$ with respect to the intensity $\phi$}
\label{tab3}
 \begin{tabular}{lccccc}
\hline\hline
\noalign{\smallskip}
$\phi$ & 1.05 & 1.25 & 1.5 & 1.75 & 2  \\\noalign{\smallskip}
$b^*_1$ & 0.37 & 1.82 & 3.81 & 6.05 & 8.16 \\\noalign{\smallskip}
$b^*_2$  & 0.35 & 1.86 & 3.97 & 6.33 & 8.51 
\\\noalign{\smallskip}
\hline
\end{tabular}   
\end{table}

Table \ref{tab4} indicates that the optimal dividend barrier $\overrightarrow{b^*}=(b^*_1,b^*_2)$ tends to increase as the intensity $p_1$ of the Poisson process $N^1 =(N^1_t)_{\geq0}$ increases from $0.1$ to $0.2$. This can be explained as follows: (i) a rise in the intensity $p_1$ implies that the surplus process grows upward faster, thereby, the company is more confident of its financial situation; (ii) due to the reason stated in (i),  if we increase the dividend barrier slightly, the expected discounted accumulated dividends decreases slightly, while the expected discounted accumulated costs of capitals injected can decrease substantially; (iii) Table \ref{tab4} reveals that,
compared to the small marginal effect on  (reduction of) the accumulated discounted dividends,
increasing the dividend barrier slightly
has a big marginal effect on (reduction of) the accumulated discounted costs of injected capitals; (iv) the company, after weighing cons and pros, choose the option of increasing the dividend barrier $b_{1}^{*}$ slightly from $3.50$ to $3.81$, and, increasing the dividend barrier $b_{2}^{*}$ slightly from $3.85$ to $3.97$. 

\begin{table}[H]
\centering
\caption{$\ora{b^*}$ with respect to the intensity $p_1$}
\label{tab4}
 \begin{tabular}{lccccc}
\hline\hline
\noalign{\smallskip}
$p_1$  & 0.1 &0.12 & 0.15& 0.17 & 0.2\\\noalign{\smallskip}
$b^*_1$  & 3.50 &3.57 & 3.65 &3.71 & 3.81  \\\noalign{\smallskip}
$b^*_2$  & 3.85 &3.89 & 3.91 &3.93 & 3.97 
\\\noalign{\smallskip}
\hline
\end{tabular}   
\end{table}

Table \ref{tab5} states that the optimal dividend barrier $\overrightarrow{b^*}=(b^*_1,b^*_2)$ tends to increase as the drift coefficient $\mu_1$  decreases from $-1.4$ to $-1.8$. This is natural, because: (i) a decrease in the drift coefficient $\mu_1$ means a sharper downward trend of the surplus process (the company's financial situation is getting worse), which implies that more capitals need to be  injected; (ii)
to reduce the accumulated discounted costs of injected capitals, so that the supremum in \eqref{eq:valuefcn} can be achieved, the company has to increase the dividend barriers. By doing this, the company secures its financial
situation and partially reduces the chance of needing capital injection.

\begin{table}[H]
\centering
\caption{$\ora{b^*}$ with respect to  $\mu_1$}
\label{tab5}
 \begin{tabular}{lccccc}
\hline\hline
\noalign{\smallskip}
$\mu_1$ &-1.4  & -1.5 & -1.6 & -1.7 & -1.8 \\\noalign{\smallskip}
$b^*_1$ &3.77 & 3.81 & 4.06 & 4.08 & 4.10 \\\noalign{\smallskip}
$b^*_2$ &3.86  & 3.97 & 4.10 & 4.18 & 4.26 
\\\noalign{\smallskip}
\hline
\end{tabular}   
\end{table}

Table \ref{tab6} states that the optimal dividend barrier $\overrightarrow{b^*}=(b^*_1,b^*_2)$ tends to decrease as $\eta_1$  increases from $0.05$ to $0.25$.
This phenomenon may be explained as follows: (i) a increase in $\eta_{1}$ implies a decrease in the magnitude of the upward jump size of $X^{1}$, thereby, the corresponding surplus process grows upward more slowly and is less likely to achieve high levels; (ii) due to the reason stated in (i), the company has to lower the dividend barrier to guarantee that the accumulated discounted dividends would not be reduced severely.

\begin{table}[H]
\centering
\caption{$\ora{b^*}$ with respect to  $\eta_1$}
\label{tab6}
 \begin{tabular}{lccccc}
\hline\hline
\noalign{\smallskip}
$\eta_1$ &0.05 &0.1 & 0.15& 0.2 & 0.25  \\\noalign{\smallskip}
$b^*_1$ &7.60 & 5.86 & 4.63 & 3.81 & 3.17  
\\\noalign{\smallskip}
$b^*_2$& 7.46 & 5.88 &4.74 &3.97 &  3.36
\\\noalign{\smallskip}
\hline
\end{tabular}   
\end{table}

}

\appendix
\section{Preliminaries of Spectrally Positive L\'{e}vy Processes}\label{sec:AppendixA}

In this appendix, we provide some preliminaries of spectrally positive L\'{e}vy processes (SPLP), which can be used to describe the main results in the main body of the paper.

Let $X=(X_t)_{t\geq 0}$ be a L\'{e}vy process defined on the filtered probability space $(\Omega, \mathcal{F},\mathbb{F},\mathbb{P})$ with the filtration $\Fx=(\F_t)_{t\geq0}$ satisfying the usual conditions. For $x\in\mathbb{R}$, we define $\Px_x$ as the law of $X$ starting from $x$, and denote by $\Ex_x$ the associated expectation under $\Px_x$. For notation simplicity, we use $\mathbb{P}$ and $\mathbb{E}$ rather than $\mathbb{P}_0$ and $\mathbb{E}_0$. The L\'{e}vy process $X$ is called a spectrally positive one if it has no downward jumps. To avoid triviality, we assume that $X$ is not a subordinator. The Laplace exponent of $X$, i.e., $\psi: [0,\infty)\rightarrow \mathbb{R}$ satisfying
${\Ex}[e^{-\theta X_t}]=:e^{\psi(\theta)t}$ for all $ t,\theta\geq 0$, is defined through the L\'{e}vy-Khintchine formula as follows
\begin{align}\label{eq:psiLevy}
\psi(\theta):=c \theta+\frac{\sigma^2}{2}\theta^2+\int_{(0,\infty)}(e^{-\theta z}-1+\theta z\mathbf{1}_{\{z<1\}})\nu(dz), \quad \forall\theta\geq 0,
\end{align}
where $c\in\mathbb{R}$, $\sigma\geq 0$, and $\nu$ is the L\'{e}vy measure of $X$ with support $(0,\infty)$ such that $\int_{(0,\infty)}(1\wedge z^2)\nu(dz)<\infty$. 

It is known that $X$ has paths of bounded variation if and only if $\sigma=0$ and $\int_{(0,1)}z\nu(dz)<\infty$; in this case, we have $X_t=-ct+S_t$ for all $t\geq0$ with $\overline{c}:=c+\int_{(0,1)}z\nu(dz)$ and $(S_t)_{t\geq 0}$ being a drift-less purely jump subordinator.
As we have already ruled out the case that $X$ has monotone paths, it holds that $c>0$. Its Laplace exponent is $\psi(\theta)=c\theta+\int_{(0,\infty)}(e^{-\theta z}-1)\nu(dz)$ for all $\theta\geq0$.

Let us also recall the $q$-scale function associated with the spectrally positive L\'{e}vy process $X$. For $q>0$, there exists a unique continuous and increasing function $W_{q}:\mathbb{R}\rightarrow [0,\infty)$, called the $q$-scale function, that vanishes on $(-\infty,0)$ and has Laplace transform over $[0,\infty)$ given by
\begin{align}\label{eq:LTWq}
\int_0^{\infty} e^{-sx}W_{q}(x)dx=\frac{1}{\psi(s)-q},\quad \forall s>\Phi_{q}:=\sup\{s\geq 0;~\psi(s)=q\}.
\end{align}
We also define that
\begin{align}\label{eq:ZqLevy}
Z_{q}(x):=1+q\int_0^x W_{q}(y)dy,\quad \forall x\in\mathbb{R},
\end{align}
and its anti-derivative is given by
\begin{align}\label{eq:anti-deriZq}
\overline{Z}_{q}(x):=\int_0^x Z_{q}(y)dy,\quad \forall x\in\mathbb{R}.
\end{align}
The so-called second scale function is defined by, for $s\geq0$,
\begin{align}\label{eq:Zqxs}
Z_q(x,s)=e^{sx}\left[1-(\psi(s)-q)\int_0^xe^{-sy}W_q(y) d y\right],\quad \forall x\geq0,
\end{align}
and $Z_q(x,s)=e^{sx}$ for all $x<0$. Note that $Z_q(x,s)=Z_q(x)$ for $s=0$ and that we can rewrite $Z_q(x,s)$ for all $s>\Phi_q$ in the form
\begin{eqnarray}
Z_q(x,s)=(\psi(s)-q)\int_0^{\infty}e^{-sy}W_q(x+y) d y,\quad \forall x\geq0,\,s>\Phi_q.\nonumber
\end{eqnarray}

We also recall that, in case that $X$ has paths of bounded variation (BV), $W_q(x)\in C^1(\R_+)$ if and only if the L\'{e}vy measure $\nu$ has no atoms. However, if $X$ has paths of unbounded variation  (NBV), $W_q(x)\in C^1(\R_+)$. Further more, if $\sigma>0$, $W_q(x)\in C^2(\R_+)$. Thus $Z_q(x)\in C^1(\R_+)$, $\overline{Z}_q(x)\in C^1(\mathbb{R})\cap C^2(\R_+)$ when $X$ has paths of bounded variation; and  $Z_q(x)\in C^1(\mathbb{R})$, $\overline{Z}_q(x) \in C^2(\mathbb{R})\cap C^3(\R_+)$ when $X$ has paths of unbounded variation. In addition, one has
\begin{align*}
\begin{split}
W_{q} (0+) &= \left\{ \begin{array}{ll}
~0, & \textrm{if $X$ is of NBV;} \\ 1/\overline{c}, & \textrm{if $X$ is of BV.}
\end{array} \right.
\end{split}
\end{align*}

For {\color{black}$a,b\in\R$}, define $\tau_{a}^{-}:=\inf\{t\geq 0;~X_{t}<a\}$ and $\tau_{b}^{+}:=\inf\{t\geq 0;~X_{t}>b\}$. Then, for any $b>0$ and $x\in[0,b]$, we have
\begin{align}
\Ex_{x}\left[e^{-q \tau_{0}^{-}}\mathbf{1}_{\{\tau_{0}^{-}<\tau_{b}^{+}\}}\right]&=\frac{W_{q}(b-x)}{W_{q}(b)},\label{fluc.1}\\
\Ex_{x}\left[e^{-q \tau_{b}^{+}}\mathbf{1}_{\{\tau_{b}^{+}<\tau_{0}^{-}\}}\right]&=
Z_{q}(b-x)-\frac{Z_{q}(b)}{W_{q}(b)}W_{q}(b-x).\label{fluc.2}
\end{align}

We next recall briefly some concepts of the excursion theory for the SPLP $X$ reflected from the infimum, i.e., $(X_t-\underline{X}_t)_{t\geq0}$ with $\underline{X}_t:=\inf_{0\leq s\leq t}X_{s}$, and the reader refer to \cite{Bertoin96} in detail. For $x\in\R$, the process $L=(L_t:= -\underline{X}_t+x)_{t\geq0}$ represents the local time at $0$ of
the Markov process $(X_t-\underline{X}_t)_{t\geq0}$ under $\mathbb{P}_{x}$. Define the corresponding inverse local time as
\begin{align}
  L_t^{-1}:=\inf\{s\geq0;~ L_s>t\},\quad \forall t\geq0.
\end{align}
Set $L_{t-}^{-1}:=\lim\limits_{s\rightarrow t-}L_s^{-1}$ be the left limit of $L_s^{-1}$ at $s=t$. If $\Delta L_t^{-1}:=L_t^{-1}-L_{t-}^{-1}>0$, we then define the following Poisson point process {\color{black}$(t, e_{t}(\cdot))_{t\geq0}$} as follows:
\begin{align*}
    e_{t}(s):=X(L_{t-}^{-1}+s)-X(L_t^{-1}),\quad \forall s\in(0,\Delta L_t^{-1}],
\end{align*}
and $\Delta L_t^{-1}$ is referred to as the lifetime of $e_{t}$. However, when $L_t^{-1}-L_{t-}^{-1}=0$, we define $e_{t}:=\Upsilon$ with $\Upsilon$ is some additional isolated point.

A conclusion drawn by It\^{o} (see e.g. Theorem 6.14 in \cite{Kyprianou14}) states that
the process $e=(t, e_t(\cdot))_{t\geq0}$ is a Poisson point process with characteristic measure $n$
when the reflected process $(X_t-\underline{X}_t)_{t\geq0}$ is recurrent; otherwise $\{e_{t}; t\leq L_{\infty}\}$ is a Poisson point process which stops until an excursion with infinite lifetime takes place. Here, $n$ is a measure on the space  $\mathcal{E}$ of excursions,
that is, the space $\mathcal{E}$  of c\`{a}dl\`{a}g functions $f$ such that
\begin{align*}
&f:\,(0,\zeta)\rightarrow (0,\infty)\,\quad \exists~  \zeta=\zeta(f)\in(0,\infty];~~f:\{\zeta\}\to\R_+,\quad \mbox{if }\zeta<\infty,
\end{align*}
where $\zeta=\zeta(f)$ represents the lifetime of the excursion; and we refer to Definition 6.13 of \cite{Kyprianou14} for details on the space $\mathcal{E}$ of canonical excursions.
Denote by $\varepsilon$ for short, a generic excursion in $\mathcal{E}$. Denote by $\overline{\varepsilon}:=\sup\limits_{t\in[0,\zeta]}\varepsilon_t$ the excursion height of $\varepsilon$. Moreover, define
\begin{align}\label{eq:rhomease}
\rho_{b}^{+}\equiv\rho_{b}^{+}(\varepsilon) :=\inf\{t\in[0,\zeta];~\varepsilon_t>b\},
\end{align}
i.e., the first passage time of $\varepsilon$ with the convention of $\inf\emptyset:=\zeta$.
In addition, let $\varepsilon_{g}$ be
the excursion (away from $0$)  with left-end point $g$ for the reflected process $(X_t-\underline{X}_t)_{t\geq0}$, and let $\zeta_{g}$ and $\overline{\varepsilon}_{g}$ be, respectively, the lifetime and excursion height of $\varepsilon_{g}$; c.f. Section IV.4 of \cite{Bertoin96}.

\section{Proof of Proposition~\ref{V.x}}\label{sec:AppendixB}

We decompose the proof of Proposition~\ref{V.x} into several lemmas. The following Lemma \ref{lem.D.R}, which gives an expression for the first term on the r.h.s. of \eqref{eq:JPhixpib} where $J_{\Psi}(x;\pi^b)$ is defined, {\color{blue}can be found in Corollaries 3.2 (iii) and 3.3 of \cite{Avram18}.} 
\begin{lemma}\label{lem.D.R}
For any $x>0$, we have
\begin{align}\label{exp.dif.div.cap.}
\Ex_x\left[\int_0^{\infty}e^{-qt}(dD^{0,b}_t-\phi dR_t^{0,b})\right]
&=
-\frac{\gamma}{q+\gamma}\left[\overline{Z}_q(b-x)+\frac{\psi^{\prime}(0+)}{q}\right]
\nonumber\\
&+\frac{\left(\gamma Z_q(b)-\phi(q+\gamma)\right)\left[Z_q(b-x,\Phi_{q+\gamma})+\frac{\gamma}{q}Z_q(b-x)\right]}{(q+\gamma)\Phi_{q+\gamma}Z_q(b,\Phi_{q+\gamma})}.
\end{align}
\end{lemma}

To obtain the expression of the second term in r.h.s. of \eqref{eq:JPhixpib} on $J_{\Psi}(x;\pi^b)$, we need to find an expression for the potential measure of SPLP controlled by a double barrier  Poissonian dividend and capital injection strategy. To this end, we need to make some preparations in the following Lemmas \ref{lem2.1}-\ref{2.2}. The following lemma computes an integral w.r.t. the excursion measure $n$ (c.f. Appendix~\ref{sec:AppendixA}), where the integrand involves the scale function, the excursion height as well as the first passage time of excursion.
\begin{lemma}\label{lem2.1}
It holds that
\begin{align*}
&n\left(e^{-p \rho_{x}^{+}(\varepsilon)}
W_{q}(x-\epsilon(\rho_{x}^{+})+y)
\mathbf{1}_{\{\overline{\epsilon}\geq x\}}\right)\nonumber\\
&\qquad=
-W_{p}(x)\frac{d}{dx}\left[\frac{W_{p}(x+y)+(q-p)\int_{0}^{y}W_{p}(x+y-z)W_{q}(z)dz}{W_{p}(x)}\right], \quad x>0,
\end{align*}
where, $\rho_{x}^{+}(\varepsilon)$ is defined by \eqref{eq:rhomease} with $\varepsilon$ being a generic excursion and $\overline{\varepsilon}:=\sup_{t\in[0,\zeta]}\varepsilon_t$ (c.f. Appendix~\ref{sec:AppendixA}).
\end{lemma}

\begin{proof}
For $0\leq a\leq x\leq c$ and $y\geq0$, it follows from \eqref{fluc.1} that
\begin{align*}
\frac{{W}_{q}(x+y)}{{W}_{q}(c+y)}&=
\Ex_{-x-y}\left[e^{-q\tau_{-c-y}^{-}}\mathbf{1}_{\{\tau_{-c-y}^{-}<\tau_{0}^{+}\}}\right]=\Ex_{-x-y}\left[e^{-q\tau_{-c-y}^{-}}\mathbf{1}_{\{\tau_{-c-y}^{-}<\tau_{-a-y}^{+}\}}\right]\nonumber\\
&\quad
+\Ex_{-x-y}\left\{\Ex_{-x-y}\left[\left.e^{-q\tau_{-c-y}^{-}}\mathbf{1}_{\{\tau_{-a-y}^{+}<\tau_{-c-y}^{-}<\tau_{0}^{+}\}}\right|\mathcal{F}_{\tau_{-a-y}^{+}}\right]\right\}\nonumber\\
&=\frac{{W}_{q}(x-a)}{{W}_{q}(c-a)}
+\Ex_{-x-y}\left[e^{-q\tau_{-a-y}^{+}}\mathbf{1}_{\{\tau_{-a-y}^{+}<\tau_{-c-y}^{-}\}}
\frac{{W}_{q}(-X_{\tau_{-a-y}^{+}})}{{W}_{q}(c+y)}\right].
\end{align*}
This is equivalent to
\begin{align*}
\Ex_{-x}\left[e^{-q\tau_{-a}^{+}}
{W}_{q}(-X_{\tau_{-a}^{+}}+y)\mathbf{1}_{\{\tau_{-a}^{+}<\tau_{-c}^{-}\}}\right]&={W}_{q}(x+y)-\frac{{W}_{q}(x-a)}{{W}_{q}(c-a)}{W}_{q}(c+y).
\end{align*}
Then, by Lemma 2.1 of \cite{Loeffen14}, we have
\begin{align*}
&\Ex_{-x}\left[e^{-p\tau_{-a}^{+}}
{W}_{q}(-X_{\tau_{-a}^{+}}+y)\mathbf{1}_{\{\tau_{-a}^{+}<\tau_{-c}^{-}\}}\right]=
{W}_{q}(x+y)-(q-p)\int_{a}^{x}{W}_{p}(x-z){W}_{q}(z+y)dz\nonumber\\
&\qquad-\frac{{W}_{p}(x-a)}{{W}_{p}(c-a)}\left[{W}_{q}(c+y)-(q-p)\int_{a}^{c}{W}_{p}(c-z)W_{q}(z+y)dz\right].
\end{align*}
In particular, we have
\begin{align}\label{2.10.0}
&\Ex_{-x}\left[e^{-p\tau_{0}^{+}}
{W}_{q}(-X_{\tau_{0}^{+}}+y)\mathbf{1}_{\{\tau_{0}^{+}<\infty\}}\right]=
{W}_{q}(x+y)-(q-p)\int_{0}^{x}{W}_{p}(x-z){W}_{q}(z+y)dz
\nonumber\\
&\qquad\qquad-\lim\limits_{c\uparrow\infty}\frac{{W}_{p}(x)}{{W}_{p}(c)}
\left[{W}_{p}(c+y)+(q-p)\int_{0}^{y}{W}_{p}(c+y-z){W}_{q}(z)dz\right]\nonumber\\
&\qquad={W}_{p}(x+y)+(q-p)\int_{0}^{y}{W}_{p}(x+y-z){W}_{q}(z)dz
\nonumber\\
&\qquad\qquad-{W}_{p}(x)
\left[e^{\Phi_{p}y}+(q-p)\int_{0}^{y}e^{\Phi_{p}(y-z)}{W}_{q}(z)dz\right],
\end{align}
where we used the fact that $\lim\limits_{z\to\infty}\frac{{W}_{q}(x+z)}{{W}_{q}(z)}=e^{\Phi_{q}x}$, and that, for any $p,q\geq0$, $x>0$ and $x+y\geq0$,
\begin{align}\label{3.17}
&{W}_{q}(x+y)-(q-p)\int_{0}^{x}{W}_{q}(z+y){W}_{p}(x-z) d z\nonumber \\
&\qquad =
{W}_{p}(x+y)+(q-p)\int_{0}^{y}{W}_{p}(x+y-z){W}_{q}(z)dz,
\end{align}
which can be verified by taking Laplace transform on both sides of the above equation. By the compensation formula, it follows that
{\small\begin{align*}
&\Ex_{-x}\left[e^{-p \tau_{0}^{+}}W_{q}(-X_{\tau_{0}^{+}}+y)\mathbf{1}_{\{\tau_{0}^{+}<\infty\}}\right]
\nonumber\\
&=\Ex_{-x}\left[\sum_{g}e^{-pg}\prod\limits_{h<g}\mathbf{1}_{\{\overline{\epsilon}_{h}<x+{\color{black}L(h)}\}}e^{-p \rho_{x+L(g)}^{+}(\epsilon_{g})}W_{q}(
x+L(g)-\epsilon_{g}(\rho_{x+L(g)}^{+})
+y)\mathbf{1}_{\{\overline{\epsilon}_{g}\geq x+L(g)\}}\right]
\nonumber\\
&=\Ex_{-x}\left[\int_{0}^{\infty}e^{-p w}\prod\limits_{h<w}\mathbf{1}_{\{\overline{\epsilon}_{h}
<x+L(h)\}}
\int_{\mathcal{E}}e^{-p \rho_{x+L(w)}^{+}(\varepsilon)}
W_{q}(x+L(w)-\epsilon(\rho_{x+L(w)}^{+})+y)\mathbf{1}_{\{\overline{\epsilon}
\geq x+L(w)\}}n(d\varepsilon)L(dw)\right]
\nonumber\\
&=\Ex_{-x}\left[\int_{0}^{\infty}e^{-p L^{-1}(w-)}
\prod\limits_{h<L^{-1}(w-)}\mathbf{1}_{\{\overline{\epsilon}_{h}
<x+L(h)\}}\int_{\mathcal{E}}
e^{-p \rho_{x+w}^{+}(\varepsilon)}
W_{q}(x+w-\epsilon(\rho_{x+w}^{+})+y)
\mathbf{1}_{\{\overline{\epsilon}\geq x+w\}}n( d \varepsilon)dw\right]
\nonumber\\
&=
\int_{x}^{\infty}\Ex_{-x}\left(e^{-p \tau_{-w}^{-}}\mathbf{1}_{\{\tau_{-w}^{-}<\tau_{0}^{+}\}}\right)n\left(e^{-p \rho_{w}^{+}(\varepsilon)}W_{q}(w-\epsilon(\rho_{w}^{+})+y)
\mathbf{1}_{\{\overline{\epsilon}\geq w\}}\right)dw
\nonumber\\
&=\int_{x}^{\infty}
\frac{W_{p}(x)}{W_{p}(w)}n\left(e^{-p \rho_{w}^{+}(\varepsilon)}
W_{q}(w-\epsilon(\rho_{w}^{+})+y)\mathbf{1}_{\{\overline{\epsilon}\geq w\}}\right)dw,
\end{align*}}
which together with \eqref{2.10.0} yields the desired result.
\end{proof}

The following Lemma computes an expectation with the integrand involving the scale function, the surplus process with capital injection $\widetilde{Y}$ and the first passage time of $\widetilde{Y}$.

\begin{lemma}\label{2.2}
We have
\begin{align*}
&\Ex_{x}\left[e^{-q\sigma_{b}^{+}}W_{q+\gamma}(b-{\widetilde{Y}}_{{\sigma}_{b}^{+}}+y)\mathbf{1}_{\{\sigma_{b}^{+}<\infty\}}\right]=
W_q(b-x+y)+\gamma\int_0^yW_q(b-x+y-z)W_{q+\gamma}(z)dz
\nonumber\\
&\qquad-\frac{W_q(b-x)}{W_q^{\prime+}(b)}\left[W^{\prime+}_q(b+y)+\gamma\int_0^yW^{\prime+}_q(b+y-z)W_{q+\gamma}(z)dz\right],\quad y\in\R_+,\,\,x\in(0,b).
\end{align*}
\end{lemma}

\begin{proof}
Put $\xi_{b}(z)=z\wedge 0+b$ and $\overline{\xi_{b}}(z)=\xi_{b}(z)-z$. Recall that
$\widetilde{Y}_t=X_{t}-\underline{X}_{t}\wedge0$ and
${\sigma}_{b}^{+}=\inf\{t\geq0; \widetilde{Y}_t>b\}=\inf\{t\geq0;~X_{t}>\xi_{b}(\underline{X}_{t})\}$.
Adapting Equation (1) in \cite{Wang21}, one can check that, for all $x\in(0,b)$ and $w>0$,
\begin{align}\label{tsei.}
\mathbb{E}_{x}\left[e^{-q\tau_{x-w}^{-}}
\mathbf{1}_{\{\tau_{x-w}^{-}<\sigma_{b}^{+}\}}\right]
=\exp\left(-\int_{x-w}^{x}\frac{W_{q}^{\prime+}(\overline{\xi_{b}}\left(z\right))}
{W_{q}(\overline{\xi_{b}}\left(z\right))}dz\right).
\end{align}
Combining \eqref{tsei.}, Lemma \ref{lem2.1} as well as the compensation formula yields that
\begin{align*}
&\Ex_{x}\left[e^{-q\sigma_{b}^{+}}W_{q+\gamma}(b-{\widetilde{Y}}_{{\sigma}_{b}^{+}}+y)\mathbf{1}_{\{\sigma_{b}^{+}<\infty\}}\right]=\mathbb{E}_{x}\Bigg[\sum_{g}e^{-q g}\prod\limits_{h<g}\mathbf{1}_{\{\overline{\epsilon}_{h}
<\overline{\xi_{b}}(x-L(h))\}}
e^{-q \rho_{\overline{\xi_{b}}(x-L(g))}^{+}(\epsilon_{g})}
\nonumber\\
&\times
W_{q+\gamma}(b+((x-L(g))\wedge 0)-
(x-L(g))-\epsilon_{g}(\rho_{\overline{\xi_{b}}(x-L(g))}^{+})
+y)\mathbf{1}_{\{\overline{\epsilon}_{g}\geq \overline{\xi_{b}}(x-L(g))\}}\Bigg]\nonumber\\
&=
\mathbb{E}_{x}\Bigg[\int_{0}^{\infty}e^{-q w}\prod\limits_{h<w}\mathbf{1}_{\{\overline{\epsilon}_{h}
<\overline{\xi_{b}}(x-L(h))\}}
\int_{\mathcal{E}}e^{-q \rho_{\overline{\xi_{b}}(x-L(w))}^{+}(\varepsilon)}\nonumber\\
&\times W_{q+\gamma}(
\overline{\xi_{b}}(x-L(w))-\epsilon_{g}(\rho_{\overline{\xi_{b}}(x-L(w))}^{+})+y)\mathbf{1}_{\{\overline{\epsilon}
\geq \overline{\xi_{b}}(x-L(w))\}}\,n(\, d \varepsilon)dL(w)\Bigg]
\nonumber\\
&=\mathbb{E}_{x}\Bigg[\int_{0}^{\infty}e^{-q L^{-1}(w-)}
\prod\limits_{h<L^{-1}(w-)}\mathbf{1}_{\{\overline{\epsilon}_{h}
<\overline{\xi_{b}}(x-L(h))\}}\nonumber\\
&\times\int_{\mathcal{E}}
e^{-q \rho_{\overline{\xi_{b}}(x-w)}^{+}(\varepsilon)}
W_{q+\gamma}(\overline{\xi_{b}}(x-w)-\epsilon(\rho_{\overline{\xi_{b}}(x-w)}^{+})+y)
\mathbf{1}_{\{\overline{\epsilon}
\geq \overline{\xi_{b}}(x-w)\}}n(d\varepsilon)dw\Bigg]
\nonumber\\
&=\int_{0}^{\infty}\mathbb{E}_{x}\left[e^{-q\tau_{x-w}^{-}}
\mathbf{1}_{\{\tau_{x-w}^{-}<\sigma_{b}^{+}\}}\right]
n\left(e^{-q \rho_{\overline{\xi_{b}}(x-w)}^{+}(\varepsilon)}
W_{q+\gamma}(\overline{\xi_{b}}(x-w)-\epsilon(\rho_{\overline{\xi_{b}}(x-w)}^{+})+y)
\mathbf{1}_{\{\overline{\epsilon}
\geq \overline{\xi_{b}}(x-w)\}}\right)dw\nonumber\\
&=\int_{-\infty}^{x}\exp\left(-\int_{w}^{x}\frac{W_{q}^{\prime+}(\overline{\xi_{b}}\left(z\right))}
{W_{q}(\overline{\xi_{b}}\left(z\right))}dz\right)
n\left(e^{-q\rho_{\overline{\xi_{b}}(w)}^{+}(\varepsilon)}
W_{q+\gamma}(\overline{\xi_{b}}(w)-\epsilon(\rho_{\overline{\xi_{b}}(w)}^{+})+y)
\mathbf{1}_{\{\overline{\epsilon}
\geq \overline{\xi_{b}}(w)\}}\right)dw\nonumber\\
&=-\int_{-\infty}^{x}
\exp\left(-\int_{w}^{x}\frac{W_{q}^{\prime+}(\overline{\xi_{b}}\left(z\right))}
{W_{q}(\overline{\xi_{b}}\left(z\right))}dz\right)
\nonumber\\
&\times
W_{q}(\overline{\xi_{b}}(w))\left.\frac{d}{dv}\left[\frac{W_{q}(v+y)+\gamma\int_{0}^{y}W_{q}(v+y-z)W_{q+\gamma}(z)dz}
{W_{q}(v)}\right]\right|_{v=\overline{\xi_{b}}(w)}dw
\nonumber\\
&=-\int_{0}^{x}
\exp\left(-\int_{w}^{x}\frac{W_{q}^{\prime+}(b-z)}
{W_{q}(b-z)}dz\right)\nonumber\\
&\times
W_{q}(b-w)\left.\frac{d}{dv}\left[\frac{W_{q}(v+y)+\gamma\int_{0}^{y}W_{q}(v+y-z)W_{q+\gamma}(z)dz}
{W_{q}(v)}\right]\right|_{v=b-w}dw
\nonumber\\
&-\int_{-\infty}^{0}
\exp\left(-\int_{0}^{x}\frac{W_{q}^{\prime+}(b-z)}
{W_{q}(b-z)}dz\right)
\exp\left(-\int_{w}^{0}\frac{W_{q}^{\prime+}(b)}
{W_{q}(b)}dz\right)
\nonumber\\
&\times
W_{q}(b)\left.\frac{d}{dv}\left[\frac{W_{q}(v+y)+\gamma\int_{0}^{y}W_{q}(v+y-z)W_{q+\gamma}(z)dz}
{W_{q}(v)}\right]\right|_{v=b}dw\nonumber\\
&=-W_q(b-x)\int_{-x}^{0}
\left.\frac{d}{dv}\left[\frac{W_{q}(v+y)+\gamma\int_{0}^{y}W_{q}(v+y-z)W_{q+\gamma}(z)dz}{W_{q}(v)}\right]\right|_{v=w+b}
dw\nonumber\\
&-W_q(b-x)\int_{0}^{\infty}e^{-\frac{W_q^{\prime+}(b)}{W_q(b)}w}dw
\left.\frac{ d }{ d v}\left[\frac{W_{q}(v+y)+\gamma\int_{0}^{y}W_{q}(v+y-z)W_{q+\gamma}(z)dz}
{W_{q}(v)}\right]\right|_{v=b}\nonumber\\
&=W_q(b-x+y)+\gamma\int_0^yW_q(b-x+y-z)W_{q+\gamma}(z)dz
\nonumber\\
&\qquad-\frac{W_q(b-x)}{W_q^{\prime+}(b)}\left[W^{\prime+}_q(b+y)+\gamma\int_0^yW^{\prime+}_q(b+y-z)W_{q+\gamma}(z)dz\right].
\end{align*}
This completes the proof.
\end{proof}

Thanks to Lemma \ref{2.2}, we are now able to give the expression for the potential measure of SPLP controlled by a double barrier  Poissonian dividend and capital injection strategy in the following Lemma \ref{lem.w}.
\begin{lemma}\label{lem.w}
For any $x>0$, it holds that
\begin{align}\label{pot.mea}
&{\Px}_x(U_{e_q}^{0,b}\in dy)=
\frac{qZ_q(b-x,\Phi_{q+\gamma})+\gamma Z_q(b-x)}{\Phi_{q+\gamma}Z_q(b,\Phi_{q+\gamma})}
\Bigg[W_q(0+)\delta_0(dy)+W_q^{\prime+}(y)\mathbf{1}_{(0,b)}(y)dy
\nonumber\\
&\qquad
+\left(\gamma W_q(b)W_{q+\gamma}(y-b)+W_q^{\prime+}(y)+\gamma\int_0^{y-b}W_q^{\prime+}(y-z)W_{q+\gamma}(z)dz\right)\mathbf{1}_{(b,\infty)}(y)dy
\nonumber\\
&\qquad
-\frac{q\Phi_{q+\gamma}Z_q(b,\Phi_{q+\gamma})}{qZ_q(b-x,\Phi_{q+\gamma})+\gamma Z_q(b-x)}\bigg(\frac{\gamma}{q}W_{q+\gamma}(y-b)Z_q(b-x)+W_q(y-x)
\nonumber\\
&\qquad+\gamma\int_0^yW_q(y-x-z)W_{q+\gamma}(z)dz\bigg)\mathbf{1}_{(b,\infty)}(y)dy\Bigg]-qW_q(y-x)\mathbf{1}_{(0,b)}(y)dy.
\end{align}
Here, we recall that $e_q$ is a r.v. independent of $X$ and is exponentially distributed with mean $1/q$.
\end{lemma}

\begin{proof}
Put $\kappa_{a}^{-(+)}:=\inf\{t\geq 0; U^{0,b}_{t}<(>)a\}$ and set $g(x):={\Px}_x(U_{e_q}^{0,b}\in dy)$. Then, we have, {\color{black}for all $x>b$,}
\begin{align}\label{2.11}
g(x)&=
{\Px}_x(U_{e_q}^{0,b}\in dy;e_q<T_{1}\wedge \kappa_{b}^{-})+{\Px}_x(U_{e_q}^{0,b}\in dy;e_q>T_{1}\wedge \kappa_{b}^{-})
\nonumber\\
&={\Px}_x(U_{e_q}^{0,b}\in dy;e_q>T_{1},\kappa_{b}^{-}>T_1)+{\Px}_x(U_{e_q}^{0,b}\in dy;e_q>\kappa_{b}^{-},T_1>\kappa_{b}^{-})\nonumber\\
&\quad+\frac{q}{q+\gamma}{\Px}_x(X_{e_{q+\gamma}}\in dy;e_{q+\gamma}<\tau_b^-)
\nonumber\\
&=\frac{q}{q+\gamma}{\Px}_x(X_{e_{q+\gamma}}\in dy;e_{q+\gamma}<\tau_b^-)+\left\{\mathbb{E}_{x}\left[e^{-(q+\gamma) \tau_{b}^{-}}\right]
+\frac{\gamma}{q+\gamma}{\Ex}_{x}\left[1-e^{-(q+\gamma) \tau_{b}^{-}}\right]\right\}g(b)\nonumber\\
&=q\left[e^{\Phi_{q+\gamma}(b-x)}W_{q+\gamma}(y-b)-W_{q+\gamma}(y-x)\right]\mathbf{1}_{(b,\infty)}(y)dy+\frac{\gamma}{q+\gamma}g(b)\nonumber\\
&\quad+\frac{q}{q+\gamma}e^{\Phi_{q+\gamma}(b-x)}g(b),
\end{align}
To obtain an expression of $g(x)$ for $x\in[0,b]$,
recall $\widetilde{Y}_t=X_{t}-\underline{X}_{t}\wedge0$
and
${\sigma}_{b}^{+}=\inf\{t\geq0;~\widetilde{Y}_t>b\}$.
When $x\in(0,b)$, it is verified from Theorem 1 of \cite{Pistorius04} and \eqref{2.11} that, for all $x\in[0,b]$,
\begin{align}\label{g.1}
g(x)&=\Px_x(U_{e_q}^{0,b}\in dy;e_q<\kappa_b^+)+\Px_x(U_{e_q}^{0,b}\in dy;e_q>\kappa_b^+)\nonumber\\
&=\Px_x(\widetilde{Y}_{e_q}\in dy;e_q<\sigma_b^+)+{\Ex}_x\left[e^{-q\sigma_b^+}g(\widetilde{Y}_{\sigma_b^+})\mathbf{1}_{\{\sigma_b^+<\infty\}}\right]\nonumber\\
&=q\frac{W_{q}(b-x)W_{q}(0+)}{W_{q}^{\prime+}(b)}\delta_0(dy)+q
\left(W_{q}(b-x)\frac{W_{q}^{\prime+}(y)}{W_{q}^{\prime+}(b)}-W_{q}(y-x)\right)\mathbf{1}_{(0,b)}(y)dy\nonumber\\
&\quad+q\left\{\Ex_{x}\left[e^{-q{\sigma}_{b}^{+}}e^{\Phi_{q+\gamma}(b-{\widetilde{Y}}_{{\sigma}_{b}^{+}})}\right]W_{q+\gamma}(y-b)-\Ex_{x}\left[e^{-q{\sigma}_{b}^{+}}W_{q+\gamma}(y-{\widetilde{Y}}_{{\sigma}_{b}^{+}})\right]\right\}\mathbf{1}_{(b,\infty)}(y)dy\nonumber\\
&\quad+\frac{\gamma}{q+\gamma}\Ex_{x}[e^{-q{\sigma}_b^+}]g(b)
+\frac{q}{q+\gamma}\Ex_{x}\left[e^{-q{\sigma}_{b}^{+}}e^{\Phi_{q+\gamma}(b-{\widetilde{Y}}_{{\sigma}_{b}^{+}})}\right]g(b).
\end{align}
It follows from Lemma 3.2 in \cite{Wang21} that, for all $x\in[0,b]$,
\begin{align}\label{g.2}
&{\Ex}_{x}\left[e^{-q{\sigma}_{b}^{+}}e^{\Phi_{q+\gamma}(b-{\widetilde{Y}}_{{\sigma}_{b}^{+}})}\right]
=\int_{-x}^{0}\exp{\left(-\int_{-x}^s\frac{W_q^{\prime+}(z+b)}{W_q(z+b)} d z\right)}\nonumber\\
&\quad\times\left[\frac{W_q^{\prime+}(s+b)}{W_q(s+b)}Z_q(s+b,\Phi_{q+\gamma})
-\Phi_{q+\gamma}Z_q(s+b,\Phi_{q+\gamma})+\gamma W_q(s+b)\right]ds
\nonumber\\
&\quad+\int_{0}^{\infty}\exp{\left(-\int_{-x}^0\frac{W_q^{\prime+}(z+b)}{W_q(z+b)} d z-\int_{0}^s\frac{W_q^{\prime+}(b)}{W_q(b)} d z\right)}\nonumber\\
&\quad
\times\left(\frac{W_q^{\prime+}(b)}{W_q(b)}Z_q(b,\Phi_{q+\gamma})
-\Phi_{q+\gamma}Z_q(b,\Phi_{q+\gamma})+\gamma W_q(b)\right)ds\nonumber\\
&=Z_q(b-x,\Phi_{q+\gamma})-\frac{W_q(b-x)}{W_q^{\prime+}(b)}\left(\Phi_{q+\gamma}Z_q(b,\Phi_{q+\gamma})-\gamma W_q(b)\right).
\end{align}
Combining \eqref{g.1}, \eqref{g.2} and Lemma \ref{2.2}, it is verified that, for all $x\in[0,b]$,
\begin{align}\label{3.23}
g(x)&=q\frac{W_q(b-x)W_q(0+)}{W_q^{\prime+}(b)}\delta_0( d y)
+q\Bigg\{W_{q+\gamma}(y-b)\bigg[Z_q(b-x,\Phi_{q+\gamma})
-\frac{W_q(b-x)}{W_q^{\prime+}(b)}\nonumber\\
&\times\left(\Phi_{q+\gamma}Z_q(b,\Phi_{q+\gamma})-\gamma W_q(b)\right)\bigg]-W_q(y-x)-\gamma\int_0^{y-b}W_q(y-x-z)W_{q+\gamma}(z)dz\nonumber\\
&+\frac{W_q(b-x)}{W^{\prime+}_q(b)}\left(W_q^{\prime+}(y)+\gamma\int_0^{y-b}W_q^{\prime+}(y-z)W_{q+\gamma}(z)dz\right)\Bigg\}\mathbf{1}_{(b,\infty)}(y)dy
\nonumber\\
&+\left[\frac{\gamma}{q+\gamma}Z_q(b-x)+\frac{q}{q+\gamma}\left(Z_q(b-x,\Phi_{q+\gamma})-\Phi_{q+\gamma}Z_q(b,\Phi_{q+\gamma})\frac{W_q(b-x)}{W_q^{\prime+}(b)}\right)\right]g(b)\nonumber\\
&+q\left(W_q(b-x)\frac{W_q^{\prime+}(y)}{W_q^{\prime+}(b)}-W_q(y-x)\right)\mathbf{1}_{(0,b)}(y)dy.
\end{align}

We next consider the following sub-cases:
\begin{itemize}
    \item The case where $X$ has paths of bounded variation.
\end{itemize}
In this case, letting $x=b$ in \eqref{3.23} and then using \eqref{3.17} with $x=0$, we get
\begin{align}\label{g.b.1}
g(b)&=\left(\frac{q}{q+\gamma}\Phi_{q+\gamma}Z_q(b,\Phi_{q+\gamma})\right)^{-1}\Bigg\{qW_q(0+)\delta_0( d y)+qW_q^{\prime+}(y)\mathbf{1}_{(0,b)}(y)dy+q\Big[W_{q+\gamma}(y-b)\nonumber\\
&\quad\times\left(\gamma  W_q(b)-\Phi_{q+\gamma}Z_q(b,\Phi_{q+\gamma})\right)
+W_q^{\prime+}(y)+\gamma\int_0^{y-b}W_q^{\prime+}(y-z)W_{q+\gamma}(z) d z\Big]\mathbf{1}_{(b,\infty)}(y)dy\Big\}.
\end{align}
\begin{itemize}
    \item The case where $X$ has paths of unbounded bounded variation.
\end{itemize}
We claim that \eqref{g.b.1} remains valid in this case. In fact, by (9) in \cite{Albrecher16} and the strong Markov property, we have, for all $x\in(0,b)$,
\begin{align}\label{g.b.unb}
g(b)&=\Px_b(U_{e_q}^{0,b}\in dy;\kappa_x^-<e_q\wedge T_1)+\Px_b(U_{e_q}^{0,b}\in dy;e_q<\kappa_x^-\wedge T_1)+\Px_b(U_{e_q}^{0,b}\in dy;T_1<\kappa_x^-\wedge e_{q})\nonumber\\
&=\Ex_b\left[e^{-(q+\gamma)\tau_x^-}\right]g(x)+\frac{q}{q+\gamma}\Px_b\left[X_{e_{q+\gamma}}\in dy;e_{q+\gamma}<\tau_x^-\right]
\nonumber\\
&\quad+\frac{\gamma}{q+\gamma}g(b)\left[\Px_b(e_{q+\gamma}<\tau_x^-)-\Px_b(X_{e_{q+\gamma}}\in(x,b);\,e_{q+\gamma}<\tau_x^-)\right]
\nonumber\\
&\quad+\frac{\gamma}{q+\gamma}\int_x^bg(z)\Px_b(X_{e_{q+\gamma}}\in dz;\,e_{q+\gamma}<\tau_x^-)\nonumber\\
&=e^{-\Phi_{q+\gamma}(b-x)}g(x)+q\left(e^{-\Phi_{q+\gamma}(b-x)}W_{q+\gamma}(y-x)-W_{q+\gamma}(y-b)\right)\mathbf{1}_{(x,\infty)}(y)dy
\nonumber\\
&\quad+\frac{\gamma}{q+\gamma}g(b)\left[1-e^{-\Phi_{q+\gamma}(b-x)}-(q+\gamma)\int_0^{b-x}e^{-\Phi_{q+\gamma}(b-x)}W_{q+\gamma}(b-x-z)dz\right]
\nonumber\\
&\quad+{\gamma}\int_0^{b-x}e^{-\Phi_{q+\gamma}(b-x)}W_{q+\gamma}(b-x-z)g(b-z)dz.
\end{align}
Plugging \eqref{g.b.unb} into \eqref{3.23}, using \eqref{3.17}, and then rearranging the yielding equation gives
\begin{align}
\label{g.x.unb}
&
g(x)\left[1-\frac{\gamma Z_q(b-x)+qZ_q(b-x,\Phi_{q+\gamma})-q\Phi_{q+\gamma}Z_q(b,\Phi_{q+\gamma})\frac{W_q(b-x)}{W_q^{\prime+}(b)}}{qe^{\Phi_{q+\gamma}(b-x)}+\gamma Z_{q+\gamma}(b-x)}\right]
\nonumber\\
=&
q\frac{W_q(b-x)W_q(0+)}{W_q^{\prime+}(b)}\delta_0({d}y)+q\Big(W_q(b-x)\frac{W_q^{\prime+}(y)}{W_q^{\prime+}(b)}-W_q(y-x)\Big)\mathbf{1}_{(0,b)}(y){d}y
\nonumber\\
&
+q\frac{W_q(b-x)}{W_q^{\prime+}(b)}\Big[W_{q+\gamma}(y-b)\Big(\gamma W_q(b)-\Phi_{q+\gamma}Z_q(b,\Phi_{q+\gamma})\Big)+W_q^{\prime+}(y)
\nonumber\\
&
+\gamma\int_0^{y-b}W_q^{\prime+}(y-z)W_{q+\gamma}(z){d}z\Big]\mathbf{1}_{(b,\infty)}(y){d}y
+q\gamma\int_0^{b-x}W_{q+\gamma}(y-b+z)W_q(b-x-z){d}z
\nonumber\\
&
+\frac{q}{{q+\gamma e^{-\Phi_{q+\gamma}(b-x)}Z_{q+\gamma}(b-x)}}\bigg\{\bigg[\gamma e^{-\Phi_{q+\gamma}(b-x)}\left[\int_0^{b-x}qW_q(z){d}z-\int_0^{b-x}(q+\gamma)W_{q+\gamma}(z){d}z\right]
\nonumber\\
&
-q\gamma\int_0^{b-x}e^{-\Phi_{q+\gamma}z}W_q(z){d}z\bigg]W_{q+\gamma}(y-x)
-\bigg[\gamma\left[\int_0^{b-x}qW_q(z){d}z-\int_0^{b-x}(q+\gamma)W_{q+\gamma}(z){d}z\right]
\nonumber\\
&
+\gamma^2 Z_{q+\gamma}(b-x)\int_0^{b-x}e^{-\Phi_{q+\gamma}z}W_q(z){d}z\bigg]W_{q+\gamma}(y-b)\bigg\}\mathbf{1}_{(b,\infty)}(y){d}y
\nonumber\\
&
+\frac{\gamma Z_q(b-x)+qZ_q(b-x,\Phi_{q+\gamma})}{q+\gamma e^{-\Phi_{q+\gamma}(b-x)}Z_{q+\gamma}(b-x)}
\Big[ e^{-\Phi_{q+\gamma}(b-x)}\int_0^{b-x}\gamma g(b-z)W_{q+\gamma}(b-x-z){d}z
\nonumber\\
&
+q\left(e^{-\Phi_{q+\gamma}(b-x)}W_{q+\gamma}(y-x)-W_{q+\gamma}(y-b)\right)\mathbf{1}_{(x,b)}(y){d}y
\Big]
\nonumber\\
&
-\frac{W_q(b-x)}{W_q^{\prime+}(b)}\frac{q\Phi_{q+\gamma}Z_q(b,\Phi_{q+\gamma})}{q+\gamma e^{-\Phi_{q+\gamma}(b-x)}Z_{q+\gamma}(b-x)}\Big[ e^{-\Phi_{q+\gamma}(b-x)}\int_0^{b-x}\gamma g(b-z)W_{q+\gamma}(b-x-z){d}z
\nonumber\\
&
+q\left(e^{-\Phi_{q+\gamma}(b-x)}W_{q+\gamma}(y-x)-W_{q+\gamma}(y-b)\right)\mathbf{1}_{(x,\infty)}(y){d}y
\Big]
.
\end{align}
Recall that $W_{q}(0+)=0$ in case $X$ has paths of unbounded variation. Hence, for any $\theta\geq0$ and bounded function $h(y)$, one can verify that
\begin{eqnarray}\label{lim.1}
\lim\limits_{x\uparrow b}\frac{\int_0^{b-x}e^{-\theta z}W_q(z) d z}{W_q(b-x)}
=\lim\limits_{x\uparrow b}\frac{\int_0^{b-x}W_{q+\gamma}(z) d z}{W_q(b-x)}=\lim\limits_{x\uparrow b}\frac{\int_0^{b-x}h(z)W_{q+\gamma}(z) d z}{W_q(b-x)}
=0,
\end{eqnarray}
where, in the last equality of \eqref{lim.1} we used the fact that
\begin{eqnarray}\label{lim.2}
\lim\limits_{x\uparrow b^{-}}\frac{W_{q+\gamma}(b-x)}{W_q(b-x)}=1,
\end{eqnarray}
which can be achieved by setting $y=0$ and $p=q+\gamma$ in \eqref{3.17}.
Dividing the both sides of \eqref{g.x.unb} by $W_q(b-x)$, sending $x$ upward to $b$, and then using \eqref{lim.1}, we finally find that \eqref{g.b.1} also holds true for the case when $X$ has paths of unbounded variation. Plugging \eqref{g.b.1} into \eqref{2.11} and \eqref{3.23} yields \eqref{pot.mea}. Thus, the proof is complete.
\end{proof}

\begin{proof}[Proof of Proposition~\ref{V.x}]
    The desired result follows from Lemma \ref{lem.D.R}, Lemma \ref{lem.w}, and, the formula of integration by parts.
\end{proof}

\section{Proofs of Lemma~\ref{lem.pm}}\label{app:C}

\begin{proof}[Proof of Lemma~\ref{lem.pm}]
Let $\overline{g}(x)$ be the left hand side of \eqref{U.b}. Then, one can verify by a similar argument as used in the proof of Lemma \ref{lem.w} that, for all $x>b$,
\begin{align*}
\overline{g}(x)&=
\Px_{x}(U^{b}_{e_q}\in dy;{e_q<T_{1}\wedge \overline{\kappa}_{b}^{-}})+\Px_{x}(U^{b}_{e_q}\in dy;e_q>T_{1}\wedge\overline{\kappa}_{b}^{-},e_q<\overline{\kappa}_0^-)\nonumber\\
&=q\left[e^{\Phi_{q+\gamma}(b-x)}W_{q+\gamma}(y-b)-W_{q+\gamma}(y-x)\right]\mathbf{1}_{(b,\infty)}(y)dy
\nonumber\\
&\quad+\frac{\gamma}{q+\gamma}\overline{g}(b)+\frac{q}{q+\gamma}e^{\Phi_{q+\gamma}(b-x)}\overline{g}(b),
\end{align*}
Using a similar argument as above and \eqref{fluc.2}, one can also verify that, for all $x\in(0,b)$,
\begin{align}\label{U.b.2}
\overline{g}(x)&=
\Px_{x}(U^{b}_{e_q}\in dy;e_q< \overline{\kappa}_{b}^{+}\wedge\overline{\kappa}_0^-)+{\Px}_{x}(U^{b}_{e_q}\in dy;\overline{\kappa}_b^+<e_q<\overline{\kappa}_{0}^{-})
\nonumber\\
&=\Px_{x}(X_{e_q}\in dy;e_q< \tau_{b}^{+}\wedge\tau_0^-)+\Ex_x\left[e^{-q\tau_b^+}\mathbf{1}_{\{\tau_b^+<\tau_0^-\}}\overline{g}(X_{\tau_b^+})\right]
\nonumber\\
&=q\Bigg\{\left[Z_q(b-x,\Phi_{q+\gamma})-W_q(b-x)\frac{Z_q(b,\Phi_{q+\gamma}) }{W_q(b)}\right]W_{q+\gamma}(y-b)-W_{q+\gamma}(y-x)
\nonumber\\
&\quad+\gamma\int_0^{b-x}W_q(b-x-z)W_{q+\gamma}(y-b+z) d z+\frac{W_q(b-x)}{W_q(b)}\nonumber\\
&\qquad\qquad \times\left[W_{q+\gamma}(y)-\gamma\int_0^{b}W_q(b-z)W_{q+\gamma}(y-b+z) d z\right]\Bigg\}\mathbf{1}_{(b,\infty)}(y)dy
\nonumber\\
&\quad+\Bigg[\frac{\gamma Z_q(b-x)+qZ_q(b-x,\Phi_{q+\gamma})}{q+\gamma}-\frac{\gamma Z_q(b)+qZ_q(b,\Phi_{q+\gamma})}{q+\gamma}\frac{W_q(b-x)}{W_q(b)}\Bigg]\overline{g}(b)\nonumber\\
&\quad+q\left[\frac{W_q(b-x)}{W_q(b)}W_q(y)-W_q(y-x)\right]\mathbf{1}_{(0,b)}(y)dy,
\end{align}
When $X$ has paths of bounded variation, letting $x\uparrow b$ in \eqref{U.b.2} yields
\begin{align}\label{gb}
\overline{g}(b)&=
\left[\frac{\gamma Z_q(b)+qZ_q(b,\Phi_{q+\gamma})}{q+\gamma}\right]^{-1}\Bigg[qW_q(y)\mathbf{1}_{(0,b)}(y)dy+q\bigg(W_{q+\gamma}(y)\nonumber\\
&\quad-Z_q(b,\Phi_{q+\gamma})W_{q+\gamma}(y-b)-\gamma\int_0^bW_q(b-z)W_{q+\gamma}(y-b+z) d z\bigg)\mathbf{1}_{(b,\infty)}(y)dy\Bigg].
\end{align}
We now claim that \eqref{gb} remains valid even when $X$ has paths of unbounded variation. Using a similar argument as in the proof of Lemma \ref{lem.w}, we arrive at, for all $x\in(0,b)$,
\begin{align}\label{g.bar.b}
\overline{g}(b)&=
\Px_b(U_{e_q}^b\in dy;e_q<T_1\wedge\overline{\kappa}_x^-)+\Px_b(U_{e_q}^b\in dy;\overline{\kappa}_x^-<T_1\wedge e_q,e_q<\overline{\kappa}_0^-)
\nonumber\\
&\quad
+\Px_b(U_{e_q}^b\in dy; T_1<\overline{\kappa}_x^-\wedge e_q,e_q<\overline{\kappa}_0^-)
\nonumber\\
&=
q\left(e^{-\Phi_{q+\gamma}(b-x)}W_{q+\gamma}(y-x)-W_{q+\gamma}(y-b)\right)\mathbf{1}_{(x,\infty)}(y)dy+e^{-\Phi_{q+\gamma}(b-x)}\overline{g}(x)
\nonumber\\
&\quad
+\frac{\gamma}{q+\gamma}\overline{g}(b)\left[1-e^{-\Phi_{q+\gamma}(b-x)}-(q+\gamma)\int_0^{b-x}e^{-\Phi_{q+\gamma}(b-x)}W_{q+\gamma}(b-x-y)dy\right]
\nonumber\\
&\quad
+{\gamma}\int_0^{b-x}e^{-\Phi_{q+\gamma}(b-x)}W_{q+\gamma}(b-x-y)\overline{g}(b-y)dy.
\end{align}
Plugging \eqref{g.bar.b} into \eqref{U.b.2}, and making the same thing to the yielding equation as we did to \eqref{g.x.unb},
we finally find that \eqref{gb} also holds true for the case when $X$ has paths of unbounded variation. Plugging \eqref{gb} into \eqref{U.b.2} yields \eqref{U.b}. Thus, the proof of the lemma is complete.
\end{proof}

\section*{Data availability statement}
This paper has no associated data.

\section*{Conflict of interest}
No potential conflict of interest was reported by the authors.

{
\section*{Acknowledgements}
The authors are grateful to the anonymous referees for their very careful reading of the
paper, and for their very constructive and helpful suggestions and comments. Wenyuan Wang acknowledges the financial support from the National Natural Science Foundation of
China (Nos.: 12171405; 11661074).
}

\end{document}